\documentclass[11pt,reqno]{amsart}
\RequirePackage{amsmath}
\RequirePackage{amssymb}
\usepackage{mathrsfs}
\usepackage{bm}
\usepackage{amsthm}
\usepackage[colorlinks,linkcolor=blue,anchorcolor=red,citecolor=blue]{hyperref}
 \usepackage{geometry}
\newcommand{\qe}{}
\newtheorem{Thm}{Theorem}[section]

\newtheorem{Lem}[Thm]{Lemma}
\newtheorem{Coro}[Thm]{Corollary}

\numberwithin{equation}{section}
\newcommand{\1}{\mathbf{1}}
\newcommand{\R}{\mathbb{R}}
\newcommand{\Rd}{{\mathbb{R}^3}}

\renewcommand{\Re}{\text{Re}}

\newcommand{\N}{\mathfrak{N}}

\newcommand{\E}{\mathcal{E}}
\newcommand{\D}{\mathcal{D}}

\newcommand{\pa}{\partial}
\newcommand{\na}{\nabla}

\newcommand{\al}{\alpha}

\renewcommand{\S}{\mathbb{S}}
\newcommand{\<}{\langle}
\renewcommand{\>}{\rangle}
\newcommand{\I}{\mathbf{I}}
\newcommand{\II}{\mathbf{I}_{\pm}}
\renewcommand{\P}{\mathbf{P}}
\newcommand{\PP}{\mathbf{P}_{\pm}}

\usepackage{cite}

\allowdisplaybreaks
\newcommand{\T}{\mathcal{T}}

\newcommand{\vertiii}[1]{{\left\vert\kern-0.25ex\left\vert\kern-0.25ex\left\vert #1 \right\vert\kern-0.25ex\right\vert\kern-0.25ex\right\vert}}

\title[Regularity of VPB system for soft potential]{Global Regularity of the Vlasov-Poisson-Boltzmann System Near Maxwellian Without Angular Cutoff for Soft Potential}

\author[D.-Q. Deng]{Dingqun Deng}
\address[D.-Q. Deng]{Beijing Institute of Mathematical Sciences and Applications and Yau Mathematical Science Center, Tsinghua Univeristy, Beijing, People's Republic of China}
\email{dingqun.deng@gmail.com}
\thanks{ORCID: 0000-0001-9678-314X}
\begin{document}


\begin{abstract}
We consider the non-cutoff Vlasov-Poisson-Boltzmann (VPB) system of two species with soft potential in the whole space $\mathbb{R}^3$ when an initial data is near Maxwellian. Continuing the work Deng [{\it Comm. Math. Phys.} 387, 1603-1654 (2021)] for hard potential case, we prove the global regularity of the Cauchy problem to VPB system for the case of soft potential in the whole space for the whole range $0<s<1$. This completes the smoothing effect to the Vlasov-Poisson-Boltzmann system, which shows that any classical solutions are smooth with respect to $(t,x,v)$ for any positive time $t>0$. The proof is based on the time-weighted energy method building upon the pseudo-differential calculus. 
%In this paper we study the regularity of the non-cutoff Vlasov-Poisson-Boltzmann system of two species with soft potential in the whole space $\mathbb{R}^3$. The smoothing effect of these solutions near Maxwellian has been a challenging open problem since \cite{Guo2002}. We prove the regularizing effect to the Cauchy problem for soft potential, which shows that any classical solutions are smooth with respect to $(t,x,v)$ for any positive time $t>0$. This gives the regularity to Vlasov-Poisson-Boltzmann system, which enjoys a similar smoothing effect as Boltzmann equation. The proof is based on the time-weighted energy method building upon the pseudo-differential calculus. 
% \begin{keyword}
%Vlasov-Poisson-Boltzmann system \sep regularity\sep non-cutoff\sep regularizing effect.
%\MSC[2020]76P05, 76X05, 35Q20, 82C40.
%\end{keyword}
%\keywords{Vlasov-Poisson-Boltzmann system \and regularity\and non-cutoff\and regularizing effect.
%}
%\paragraph{Keywords} 
%\paragraph{MSC 2020}

\end{abstract}

\date{\today}
\maketitle	
\tableofcontents

\tableofcontents

\section{Introduction}
The Vlasov-Poisson-Boltzmann system is an important physical model to describe the time evolution of plasma particles of two species (e.g. ions and electrons). 
In this work we study the smoothing effect of solutions to non-cutoff Vlasov-Poisson-Boltzmann system with $-\frac{3}{2}-2s<\gamma\le -2s$ and $0<s<1$. We find that the solutions enjoy the same smoothing phenomenon as the Boltzmann equation, which gives the regularity of the Vlasov-Poisson-Boltzmann system.  Since Duan-Liu \cite{Duan2013} found the global solution for non-cutoff soft potential with $1/2\le s<1$, the smoothing effect for the VPB system is an open interesting problem. 
In \cite{Deng2021b}, the author finds out the smoothing effect for hard potential. In this work, we finally recover the smoothing effect for non-cutoff soft potential with the whole range $0<s<1$.

\subsection{Equations} We consider the Vlasov-Poisson-Boltzmann system of two species in the whole space $\R^3$, cf. \cite{Krall1973,Guo2002}:
\begin{equation}\begin{aligned}\label{1}
	\partial_tF_+ + v\cdot\nabla_xF_+ +E\cdot\nabla_vF_+ = Q(F_+,F_+) + Q(F_-,F_+),\\
	\partial_tF_- + v\cdot\nabla_xF_- -E\cdot\nabla_vF_- = Q(F_-,F_-) + Q(F_+,F_-).
\end{aligned}
\end{equation}
The self-consistent electrostatic field is taken as $E(t,x) = -\nabla_x\phi$, with the electric potential $\phi$ given by 
\begin{align}\label{2}
	-\Delta_x\phi = \int_{\Rd}(F_+-F_-)\,dv, \quad  \phi\to 0 \text{ as } |x|\to\infty.
\end{align}
The initial data of the system is 
\begin{align}\label{3}
	F_\pm(0,x,v) = F_{\pm,0}(x,v). 
\end{align}
The unknown function $F_\pm(t,x,v)\ge 0$ represents the velocity distribution for the particle with position $x\in \Rd$ and velocity $v\in\Rd$ at time $t\ge 0$. The bilinear collision term $Q(F,G)$ on the right hand side of \eqref{1} is given by 
\begin{align*}
	Q(F,G)(v) = \int_{\Rd}\int_{\S^2}B(v-v_*,\sigma)\big(F'_*G' - F_*G\big)\,d\sigma dv_*, 
\end{align*} 
where $F' = F(x,v',t)$, $G'_* = G(x,v'_*,t)$, $F = F(x,v,t)$, $G_* =G(x,v_*,t)$. Velocity pairs $(v,v_*)$ and $(v',v'_*)$ are velocities before and after binary elastic collision respectively. They are defined by 
\begin{align*}
	v' = \frac{v+v_*}{2}+\frac{|v-v_*|}{2}\sigma,\ \
	v'_* = \frac{v+v_*}{2}-\frac{|v-v_*|}{2}\sigma.
\end{align*}
This two pair of velocities satisfy the conservation law of momentum and energy:
%\begin{align*}
	$v+v_*=v'+v'_*,\ \ |v|^2+|v_*|^2=|v'|^2+|v'_*|^2.$
%\end{align*}

\subsection{Collision Kernel}
The Boltzmann collision kernel $B$ is defined as 
\begin{align*}
	B(v-v_*,\sigma) = |v-v_*|^\gamma b(\cos\theta),
\end{align*}
for some function $b$ and $\gamma$ determined by the intermolecular interactive mechanism with $\cos\theta=\frac{v-v_*}{|v-v_*|}\cdot \sigma$. Without loss of generality, we can assume $B(v-v_*,\sigma)$ is supported on $(v-v_*)\cdot\sigma\ge 0$, which corresponds to $\theta\in(0,\pi/2]$, since $B$ can be replaced by its symmetrized form $\overline{B}(v-v_*,\sigma) = B(v-v_*,\sigma)+B(v-v_*,-\sigma)$ in $Q(f,f)$.
The angular function $\sigma\mapsto b(\cos\theta)$ is not integrable on $\S^2$. Moreover, there exists $0<s<1$ such that 
\begin{align*}
\frac{1}{C}\theta^{-1-2s}\le	\sin\theta b(\cos\theta)\le C \theta^{-1-2s}\ \text{ on }\theta \in (0,\pi/2],
\end{align*}
for some $C>0$. 
It's convenient to call soft potential when $\gamma+2s< 0$, and hard potential when $\gamma+2s\ge0$. In this work, we always assume
\begin{align*}
	0< s <1, \quad -\frac{3}{2}< \gamma\le -2s.
\end{align*}

In this paper, we are going to establish the smoothing effect of the solutions to the Cauchy problem \eqref{1}, \eqref{2} and \eqref{3} of the Vlasov-Poisson-Boltzmann system near the global Maxwellian equilibrium. For global existence, Guo \cite{Guo2002} firstly investigate the hard-sphere model of the Vlasov-Poisson-Boltzmann system in a periodic box. Since then, the energy method was largely developed for the Boltzmann equation with the self-consistent electric and magnetic fields. Duan-Strain \cite{Duan2011} analyzes the optimal time decay rate for the Vlasov-Maxwell-Boltzmann system with cutoff hard potential. Guo \cite{Guo2012} gives the global existence of the Vlasov-Poisson-Landau system by using an elegant weight $e^{\pm\phi}$. Duan-Liu \cite{Duan2013} investigate the Vlasov-Poisson-Boltzmann system without angular cutoff for the case of soft potential when $1/2\le s<1$.
For smoothing effect of Boltzmann equation, since the work \cite{Alexandre2000} discover the entropy dissipation property for non-cutoff linearized Boltzmann operator, there's been many discussion in different context. See \cite{Alexandre2011,Global2019,Alexandre2013,Gressman2011a,Mouhot2007} for the dissipation estimate of collision operator, and \cite{Alexandre2011aa,Alexandre2010,Barbaroux2017a,Chen2011,Chen2012,Deng2020b} for $C^\infty$ smoothing effect for the solution to Boltzmann equation in different aspect. We refer to \cite{Chen2021,Duan2021a} for Gevrey smoothing effect for spatially inhomogeneous Boltzmann equation. Recently, the author \cite{Deng2021a,Deng2021b} establish the smoothing effect of Cauchy problem for VPB system with hard potential and VPL system for Coulomb interactions.  
These works show that the Boltzmann operator behaves locally like a fractional operator:
\begin{align*}
	Q(f,g)\sim (-\Delta_v)^sg+\text{lower order terms}.
\end{align*}
More precisely, according to the symbolic calculus developed by \cite{Global2019}, the linearized Boltzmann operator behaves essentially as 
\begin{align*}
	L \sim \<v\>^\gamma(-\Delta_v+|v\wedge\partial_v|^2+|v|^2)^s+\text{lower order terms}.
\end{align*}
We also mention \cite{Imbert2021} for global regularity of Boltzmann equation without angular cutoff. 
%However, until now, the smoothing effect of the solutions to Vlasov-Poisson-Boltzmann system remains open.

\subsection{Reformulation} We will reformulate the problem near Maxwellian as in \cite{Guo2002}. For this, we denote a normalized global Maxwellian $\mu$ by 
\begin{align*}
	\mu(v) = (2\pi)^{-3/2}e^{-|v|^2/2}. 
\end{align*}
Set $F_\pm(t,x,v)=\mu(v) +\mu^{1/2}f_\pm(t,x,v)$. Denote $f=(f_+,f_-)$ and $f_0=(f_{+,0},f_{-,0})$. Then the Cauchy problem \eqref{1}, \eqref{2} and \eqref{3} can be reformulated as 
\begin{equation}\label{7}
	\partial_tf_\pm + v\cdot\nabla_xf_\pm \pm \frac{1}{2}\nabla_x\phi\cdot vf_\pm  \mp\nabla_x\phi\cdot\nabla_vf_\pm \pm \nabla_x\phi\cdot v\mu^{1/2} - L_\pm f = \Gamma_{\pm}(f,f),
\end{equation}
\begin{equation}\label{8}
	-\Delta_x \phi = \int_{\Rd}(f_+-f_-)\mu^{1/2}\,dv, \quad \phi\to 0\text{ as }|x|\to\infty,
\end{equation}
with initial data 
\begin{align}\label{9}
	f_\pm(0,x,v) = f_{\pm,0}(x,v). 
\end{align}
The linearized operator $L=(L_+,L_-)$ and bilinear collision operator $\Gamma = (\Gamma_+,\Gamma_-)$ are given by 
\begin{equation*}
	L_\pm f = \mu^{-1/2}\Big(2Q(\mu,\mu^{1/2}f_\pm) + Q(\mu^{1/2}(f_\pm+f_\mp),\mu)\Big),
\end{equation*}
\begin{equation*}
	\Gamma_\pm(f,g) = \mu^{-1/2}\Big(Q(\mu^{1/2}f_\pm,\mu^{1/2}g_\pm) + Q(\mu^{1/2}f_\mp,\mu^{1/2}g_\pm)\Big).
\end{equation*}
For later use, we introduce the bilinear operator $\T$ by 
\begin{align*}
	\T_\beta(h_1,h_2) = \int_{\Rd}\int_{\S^2}B(v-v_*,\sigma)\partial_\beta(\mu^{1/2}_*)\big(h_1(v'_*)h_2(v')-h_1(v_*)h_2(v)\big)\,d\sigma dv_*, 
\end{align*}
for two scalar functions $h_1,h_2$, and in particular, we set $\T=\T_0$. Thus, 
\begin{equation*}
		L_\pm f = 2\T(\mu^{1/2},f_\pm) + \T(f_\pm+f_\mp,\mu^{1/2}),\\
	\end{equation*}
\begin{equation*}
		\Gamma_\pm(f,g) = \T(f_\pm,g_\pm) + \T(f_\mp,g_\pm).
\end{equation*}

\subsection{Notations}
Through the paper, $C$ denotes some positive constant (generally large) and $\lambda$ denotes some positive constant (generally small), where both $C$ and $\lambda$ may take different values in different lines. 
For any $v\in\Rd$, we denote $\<v\>=(1+|v|^2)^{1/2}$. For multi-indices $\alpha=(\alpha_1,\alpha_2,\alpha_3)$ and $\beta=(\beta_1,\beta_2,\beta_3)$, write 
\begin{align*}
	\partial^\alpha_\beta = \partial^{\alpha_1}_{x_1}\partial^{\alpha_2}_{x_2}\partial^{\alpha_3}_{x_3}\partial^{\beta_1}_{v_1}\partial^{\beta_2}_{v_2}\partial^{\beta_3}_{v_3}.
\end{align*}The length of $\alpha$ is $|\alpha|=\alpha_1+\alpha_2+\alpha_3$. 
The notation $a\approx b$ (resp. $a\gtrsim b$, $a\lesssim b$) for positive real function $a$, $b$ means there exists $C>0$ not depending on possible free parameters such that $C^{-1}a\le b\le Ca$ (resp. $a\ge C^{-1}b$, $a\le Cb$) on their domain. $\mathscr{S}$ denotes the Schwartz space. $\Re (a)$ means the real part of complex number $a$. $[a,b]=ab-ba$ is the commutator between operators. $\{a(v,\eta),b(v,\eta)\} =  \partial_\eta a_1\partial_va_2 - \partial_va_1\partial_\eta a_2$ is the Poisson bracket. $\Gamma=|dv|^2+|d\eta|^2$ is the admissible metric and $S(m)=S(m,\Gamma)$ is the symbol class. 
For pseudo-differential calculus, we write $(x,v)\in \Rd\times\Rd$ to be the space-velocity variable and $(y,\eta)\in \Rd\times\Rd$ to be the corresponding variable in frequency space (the variable after Fourier transform). The $L^2_{v,x}$ space is defined as $L^2_{v,x}=L^2(\R^3_v\times\R^3_x)$. $L^2(B_C)$ is the $L^2_v$ space on Euclidean ball $B_C$ of radius $C$ at the origin. For usual Sobolev space, we will use notation 
\begin{align*}
	\|f\|_{H^k_vH^m_x} = \sum_{|\beta|\le k,|\alpha|\le m}\|\partial^\alpha_\beta f\|_{L^2_{v,x}},
\end{align*}for $k,m\ge 0$. 
We also define the standard velocity-space mixed Lebesgue space $Z_1=L^2(\R^3_v;L^1(\R^3_x))$ with the norm
\begin{equation*}
	\|f\|_{Z_1} = \Big\|\|f\|_{L^1_x}\Big\|_{L^2_v}.
\end{equation*}
In this paper, we write Fourier transform and inverse Fourier transform on $x$ as 
\begin{align*}
	\widehat{f}(y) = \int_{\R^3}f(x)e^{-ix\cdot y}\,dx,\quad f^\vee(x)=\frac{1}{(2\pi)^3}\int_{\R^3}f(y)e^{-iy\cdot x}\,dx. 
\end{align*}

\smallskip
(i) As in \cite{Guo2003a}, the null space of $L$ is given by 
\begin{equation*}
	\ker L  = \text{span}\Big\{[1,0]\mu^{1/2},[0,1]\mu^{1/2},[1,1]v_i\mu^{1/2}(1\le i\le 3),[1,1]|v|^2\mu^{1/2}\Big\}. 
\end{equation*}
We denote $\PP$ to be the orthogonal projection from $L^2_v\times L^2_v$ onto $\ker L$, which is defined by 
\begin{equation}\label{10}
	\P f = \Big(a_+(t,x)[1,0]+a_-(t,x)[0,1]+v\cdot b(t,x)[1,1]+(|v|^2-3)c(t,x)[1,1]\Big)\mu^{1/2},
\end{equation}or equivalently by 
\begin{equation*}
	\PP f = \Big(a_\pm(t,x)+v\cdot b(t,x)+(|v|^2-3)c(t,x)\Big)\mu^{1/2}.
\end{equation*}
Then for given $f$, one can decompose $f$ uniquely as 
\begin{equation*}
	f = \P f+ (\I-\P)f. 
\end{equation*}
The function $a_\pm,b,c$ are given by 
\begin{align*}
	a_\pm &= (\mu^{1/2},f_\pm)_{L^2_v} = (\mu^{1/2},\PP f)_{L^2_v},\\
	b_j&= \frac{1}{2}(v_j\mu^{1/2},f_++f_-)_{L^2_v} = (v_j\mu^{1/2},\PP f)_{L^2_v},\\
	c&=\frac{1}{12}((|v|^2-3)\mu^{1/2},f_++f_-)_{L^2_v} = \frac{1}{6}((|v|^2-3)\mu^{1/2},\PP f)_{L^2_v}. 
\end{align*}

\smallskip
(ii) To describe the behavior of linearized Boltzmann collision operator, \cite{Alexandre2012} introduce the norm $\vertiii{f}$ while \cite{Gressman2011} introduce the norm $N^{s,\gamma}_l$. The work \cite{Global2019} give the pseudo-differential-type norm $\|(\tilde{a}^{1/2})^wf\|_{L^2_v}$. They are all equivalent and we list their results as follows. 

Let $\mathscr{S}'$ be the space of tempered distribution functions. $N^{s,\gamma}$ denotes the weighted geometric fractional Sobolev space 
\begin{align*}
	N^{s,\gamma} = \{f\in\mathscr{S}':|f|_{N^{s,\gamma}}<\infty\},
\end{align*}
with the anisotropic norm 
\begin{align*}
	|f|^2_{N^{s,\gamma}}:&=\|\<v\>^{\gamma/2+s}f\|^2_{L^2}+\int(\<v\>\<v'\>)^{\frac{\gamma+2s+1}{2}}\frac{(f'-f)^2}{d(v,v')^{d+2s}}\1_{d(v,v')\le 1},
\end{align*}with $d(v,v'):=\sqrt{|v-v'|^2+\frac{1}{4}(|v|^2-|v'|^2)^2}$. 
In order to describe the velocity weight $\<v\>$, as in \cite{Gressman2011}, we define 
\begin{align*}
	|f|^2_{N^{s,\gamma}_l} = |w_l\<v\>^{\gamma/2+s}f|^2_{L^2_v}+\int_{\Rd}dv\,w_l\<v\>^{\gamma+2s+1}\int_{\Rd}dv'\,\frac{(f'-f)^2}{d(v,v')^{d+2s}}\1_{d(v,v')\le 1},
\end{align*}which turns out to be equivalent with $|w_lf|_{N^{s,\gamma}}$. This follows from the proof of Proposition 5.1 in \cite{Gressman2011} since the $\psi$ therein has a nice support. 

\smallskip
On the other hand, as in \cite{Alexandre2012}, we define 
\begin{align*}
	\vertiii{f}^2:&=\int B(v-v_*,\sigma)\Big(\mu_*(f'-f)^2+f^2_*((\mu')^{1/2}-\mu^{1/2})^2\Big)\,d\sigma dv_*dv,
\end{align*}

\smallskip
For pseudo-differential calculus as in \cite{Global2019}, one may refer to the appendix of \cite{Deng2021b} as well as \cite{Lerner2010} for more information. Let $\Gamma=|dv|^2+|d\eta|^2$ be an admissible metric. We say that
$a\in S(\Gamma)=S(M,\Gamma)$, if for $\alpha,\beta\in \N^d$, $v,\eta\in\Rd$,
\begin{align*}
	|\partial^\alpha_v\partial^\beta_\eta a(v,\eta,\xi)|\le C_{\alpha,\beta}M,
\end{align*}with $C_{\alpha,\beta}$ a constant depending only on $\alpha$ and $\beta$. The space $S(M,\Gamma)$ endowed with the seminorms
\begin{align*}
	\|a\|_{k;S(M,\Gamma)} = \max_{0\le|\alpha|+|\beta|\le k}\sup_{(v,\eta)\in\R^{2d}}
	|M(v,\eta)^{-1}\partial^\alpha_v\partial^\beta_\eta a(v,\eta,\xi)|,
\end{align*}becomes a Fr\'{e}chet space.
Define
\begin{align}\label{11a}
	\tilde{a}(v,\eta):=\<v\>^\gamma(1+|\eta|^2+|\eta\wedge v|^2+|v|^2)^s+K_0\<v\>^{\gamma+2s}
\end{align}
to be a $\Gamma$-admissible weight, where $K_0>0$ is chosen as the following. 
Applying theorem 4.2 in \cite{Global2019} and Lemma 2.1 and 2.2 in \cite{Deng2020a}, there exists $K_0>0$ such that the Weyl quantization $\tilde{a}^w:H(\tilde{a}c)\to H(c)$ and $(\tilde{a}^{1/2})^w:H(\tilde{a}^{1/2}c)\to H(c)$ are invertible, with $c$ being any $\Gamma$-admissible metric. The weighted Sobolev space $H(c)$ is defined by 
 $H(M,\Gamma):=\{u\in\mathscr{S}':\|u\|_{H(M,\Gamma)}<\infty\}$, where
\begin{align*}
	\|u\|_{H(M,\Gamma)}:=\int M(Y)^2\|\varphi^w_Yu\|^2_{L^2}|\Gamma_Y|^{1/2}\,dY<\infty,
\end{align*}and $(\varphi_Y)_{Y\in\R^{2d}}$ is any uniformly confined family of symbols which is a partition of unity. If $a\in S(M)$ is a isomorphism from $H(M')$ to $H(M'M^{-1})$, then $(a^wu,a^wv)$ is an equivalent Hilbertian structure on $H(M)$. 
 The symbol $\tilde{a}$ is real and gives the formal self-adjointness of Weyl quantization $\tilde{a}^w$. By the invertibility of $(\tilde{a}^{1/2})^w$, we have equivalence 
\begin{align*}
	\|(\tilde{a}^{1/2})^w(\cdot)\|_{L^2_v}\approx\|\cdot\|_{H(\tilde{a}^{1/2})_v},
\end{align*}and hence we will equip $H(\tilde{a}^{1/2})_v$ with norm $\|(\tilde{a}^{1/2})^w(\cdot)\|_{L^2_v}$; see \cite[Appendix]{Deng2021b}. Also,  $$\|w_l(\tilde{a}^{1/2})^w(\cdot)\|_{L^2_v}\approx\|(\tilde{a}^{1/2})^ww_l(\cdot)\|_{L^2_v}$$ due to Lemma \ref{inverse_bounded_lemma}. 

\medskip
The three norms defined above are equivalent since for $l\in\R$, 
\begin{align*}
	\|(\tilde{a}^{1/2})^wf\|^2_{L^2_v}\approx\vertiii{f}^2\approx|f|^2_{N^{s,\gamma}}\approx (-Lf,f)_{L^2_v}+\|\<v\>^lf\|_{L^2_v},
\end{align*}which follows from \cite[eq. (2.13) and (2.15)]{Gressman2011}, \cite[Proposition 2.1]{Alexandre2012} and  \cite[Theorem 1.2]{Global2019}. An important result from \cite[Section 3]{Deng2020a} is that 
\begin{align*}
	L\in S(\tilde{a}),
\end{align*}where $S(\tilde{a})=S(\tilde{a},\Gamma)$ is the pseudo-differential symbol class; see \cite[Chap. 2]{Lerner2010}. This implies that 
\begin{align*}
	|(Lf,f)_{L^2_v}|\lesssim \|(\tilde{a}^{1/2})^wf\|_{L^2}^2. 
\end{align*}
For brevity, we denote dissipation norms 
\begin{align*}
	\|f\|_{L^2_D} = \|(\tilde{a}^{1/2})^wf\|_{L^2_v}, \quad \|f\|_{L^2_xL^2_D} = \|(\tilde{a}^{1/2})^wf\|_{L^2_xL^2_v}.
\end{align*}

In order to extract the smoothing effect on $x$, we define a symbol $\tilde{b}$ by 
\begin{align}\label{b}
	\tilde{b}(v,y) = \<v\>^{l_0}|y|^{\delta_1}, 
\end{align}
where $l_0,\delta_1$ are defined by \eqref{106b}. This symbol will help us find out the smoothing rate on spatial variable.

\subsection{Main results}
To state the result of the paper, we let $K\ge 0$ to be the total order of derivatives on $v,x$ and define the velocity weight function $w_l$ for any $l\in \R$ by 
\begin{align}\label{w}
	w_{l}(\alpha,\beta) = \<v\>^{l-p|\alpha|-q|\beta|+Kp},
\end{align}
where $p,q>0$ are given by 
\begin{align*}
	p= -\gamma-\frac{2\gamma(1-s)}{s}+1,\quad q=-\frac{2\gamma}{s}+1. 
\end{align*}
For brevity, we write $w_l=w_l(0,0)$ and $w(|\al|,|\beta|)=w(\al,\beta)$. 
In order to extract the smoothing effect, as in \cite{Deng2021b}, we define a useful coefficient  
\begin{equation}\label{psi}
\psi_k=\left\{\begin{aligned}
	1, \text{  if $k\le 0$},\\
	\psi^k, \text{ if $k> 0$}, 
\end{aligned}\right.
\end{equation}
where $\psi = 1$ in Theorem \ref{thm21} (for existence) and $\psi = t^N$ with $N=N(\alpha)>0$ large in Theorem \ref{main2} and Section \ref{sec5} (for regularity).
%We will carry $\psi$ in our calculation for the brevity of proving the smoothing effect. 
When considering $\psi=t^N$ in proving regularity, we always assume $0\le t\le 1$, since regularity is a local property. In any case, we have $\psi\le 1$. 
The motivation of this weight is that when $\psi=t^N$, the initial high-order energy functional defined in \eqref{Defe} would vanish at the initial time $t=0$. This shows that high-order energy for any $t>0$ is controlled by low-order initial energy and we obtain the regularizing effect.

Corresponding to given $f=f(t,x,v)$, we introduce the instant energy functional $\E_{K,l}(t)$
% and the instant high-order energy functional $\E^{h}_{K,l}(t)$ 
% to be functionals
satisfying the equivalent relation
\begin{align}\label{Defe}
	\E_{K,l}(t)\notag &\approx \sum_{|\alpha|\le K}\|\psi_{|\alpha|-4}\partial^\alpha E\|^2_{L^2_x}+\sum_{|\alpha|\le K}\|\psi_{|\alpha|-4}\partial^\alpha\P f\|^2_{L^2_{v,x}}\\
	&\qquad+\sum_{\substack{|\alpha|+|\beta|\le K}}\|\psi_{|\alpha|+|\beta|-4}w_l(\al,\beta)\partial^\alpha_\beta(\I-\P) f\|^2_{L^2_{v,x}}.
\end{align}
The precise definition will be given in \eqref{EE}. 
Also, we define the dissipation rate functional $\D_{K,l}$ by 
\begin{align}\label{Defd}
	\D_{K,l}(t) \notag&= \sum_{|\alpha|\le K-1}\|\psi_{|\alpha|-4}\partial^\alpha E\|^2_{L^2_x}+\sum_{1\le|\alpha|\le K}\|\psi_{|\alpha|-4}\partial^\alpha\P f\|^2_{L^2_{v,x}}\\
	&\qquad+\sum_{\substack{|\alpha|+|\beta|\le K}}\|\psi_{|\alpha|+|\beta|-4}(\tilde{a}^{1/2})^ww_l(\al,\beta)\partial^\alpha_\beta(\I-\P) f\|^2_{L^2_{v,x}}.
\end{align}
Here $E=E(t,x)$ is determined by $f(t,x,v)$ in terms of $E=-\nabla_x\phi$ and \eqref{8}. Notice that one can change the order of $(\tilde{a}^{1/2})^w$ and $w_l(\al,\beta)$ due to Lemma \ref{inverse_bounded_lemma}. 
 The main result of this paper is stated as follows.

%This gives the global existence to the Vlasov-Poisson-Boltzmann system with the optimal large time decay as in \cite{Duan2011}, where Duan and Strain discover the optimal large time decay for Vlasov-Maxwell-Boltzmann system. Notice that we only require $K\ge 4$, which improve the index $K\ge 8$ in \cite{Duan2013}. 
%In order to define the $a$ $priori$ assumption, for $0<T\le\infty$ and $t\in[0,T]$, we define the time-weighted energy norm $X(t)$ by 
%\begin{align*}
%	X(t) = \sup_{0\le\tau\le t}\E_{K,l+l_1}(\tau)+\sup_{0\le\tau\le t}(1+\tau)^{3/2}\E_{K,l}(\tau)+\sup_{0\le\tau\le t}(1+\tau)^{5/2}\E^h_{K,l}(\tau).
%\end{align*}
%Here the high-order energy functional $\E^h_{K,l}$ has time decay rate $(1+t)^{-5/2}$ while $\E_{K,l}$ has time decay rate $(1+t)^{-3/2}$. They are all optimal as in the Boltzmann equation case \cite{Strain2012} and the Vlasov-Maxwell-Boltzmann system case \cite{Duan2011}. 
%Let $\delta_0>0$ and the $a$ $priori$ assumption to be 
%\begin{align}
%	\label{priori}\sup_{0\le t\le T}X(t)\le \delta_0. 
%\end{align}
%Then we will obtain the following closed $a$ $priori$ estimate
%\begin{align*}
%	X(t)\lesssim \epsilon^2_0+X^{3/2}(t)+X^2(t).
%\end{align*}

\begin{Thm}\label{main2}Let $-\frac{3}{2}-2s<\gamma\le -2s$, $0<s<1$, $0<\tau<T\le \infty$ and $l\ge 0$. 
For any $K\ge 4$ and multi-indices $|\alpha|+|\beta|\le K$, assume $\psi=t^N$ with $N>0$ large when $|\alpha|\le 4$ and $N=N(\alpha)>0$ defined by \eqref{106a} when $|\alpha|>4$. 
Let $(f,E)$ be the solution to \eqref{7}, \eqref{8} and \eqref{9} satisfying that 
% \begin{align}
% \sup_{0\le t\le T}\|\<v\>^{C}f(t)\|_{L^2_{v,x}} <\infty.
% \end{align} 
% Then the followings holds true.
 for $n>0$, there exists $C_n>0$ such that
\begin{align}\label{19aabb}
	\sup_{0\le t\le T}\|\<v\>^{n}f(t)\|_{L^2_{v,x}} \le C_n <\infty.
\end{align}
Then the followings hold true.

\smallskip\noindent
(1) If 
\begin{align}\label{15aa}
	\epsilon_1 = (\E_{4,l}(0))^{1/2}
\end{align}
is sufficiently small, then for $|\alpha|+|\beta|\le K$, $T<\infty$,
\begin{align}\label{19a}
	\sup_{\tau\le t\le T}\Big(\|w_l(\al,\beta)\partial^\alpha_\beta f\|^2_{L^2_{v,x}}+\|\partial^\alpha\nabla_x\phi\|_{L^2_x}^2\Big)\le \epsilon^2_1C_{\tau,T,K,l},
\end{align}
where $C_{\tau,T,K,l}>0$ depends on $\tau,T,K,l$.

\smallskip\noindent (2)
There exists $C_{K,l}>0$ such that if 
%\begin{align*}
	$\E_{4,C_{K,l}}(0)$
%\end{align*}
is sufficiently small, then for $|\alpha|+|\beta|\le K$, $k\ge 0$, $T<\infty$, we have 
\begin{align}\label{19b}
	\sup_{\tau\le t\le T}\Big(\|w_l(\al,\beta)\partial^\alpha_\beta\partial^{k}_tf\|^2_{L^2_{v,x}}+\|\partial^\alpha\partial^{k}_t\nabla_x\phi\|_{L^2_x}^2\Big)\le C_{\tau,T,k,K,l}<\infty,
\end{align}where $C_{\tau,T,k,K,l}$ is a constant depending on $\tau$, $T$, $k$, $K$, $l$. 
Consequently, $f\in C^\infty(\R^+_t\times\R^3_x\times\R^3_v)$. 

\smallskip\noindent
%(2) If additionally, the solution $(f,E)$ satisfies that 
%\begin{align}\label{19aa}
%	\epsilon_0 = (\E_{4,l+l_1}(0))^{1/2}+\|w^{l_2}f_0\|_{Z_1}+\|E_0\|_{L^1_x},
%\end{align}
%is sufficiently small, where $l$ is defined by \eqref{20b}, $l_1=\frac{5(\gamma+2)}{4(1-p)\gamma}$, $l_2=\frac{5(\gamma+2)}{4\gamma}$ for soft potential $-1\le\gamma+2< 0$ and $l_1=l_2=0$ for hard potential $\gamma+2\ge 0$. Then the constants in (1) can be chosen independent of $T$ and $T$ can take the value $\infty$.
(3) If additionally, the initial data satisfies that 
\begin{align}\label{19aa}
	\epsilon_0 = (\E_{4,l+l_1}(0))^{1/2}+\|w_{l_2}f_0\|_{Z_1}+\|E_0\|_{L^1_x},
\end{align}
is sufficiently small, where $l>\max\{-\frac{3(\gamma+2s)}{4},K\}$,  $l_1=-\frac{5(\gamma+2s)}{4(1-p)}$, $l_2>-\frac{5(\gamma+2s)}{4}$ are constants. Then the constants in \eqref{19a} and \eqref{19b} can be chosen independent of $T$ and $T$ can take the value $\infty$.

\smallskip \noindent
(4) Suppose that there exists sufficiently large $C_{K,l}>0$ such that if the solution $(f,E)$ satisfies that, 
\begin{align}\label{19aab}
	\epsilon_{0,K,l} = (\E_{4,C_{K,l}+l_1}(0))^{1/2}+\|w_{l_2}f_0\|_{Z_1}+\|E_0\|_{L^1_x}
\end{align}is sufficiently small. Then the condition \eqref{19aabb} can be removed and we have \eqref{19a} and \eqref{19b}. Also, the constants in \eqref{19a} and \eqref{19b} can be chosen independent of $T$ and $T$ can take the value $\infty$.

%(1) For $|\alpha|+|\beta|\le K$, $T<\infty$,
%\begin{align}\label{19a}
%\sup_{\tau\le t\le T}\|w^{l-|\alpha|-|\beta|}\partial^\alpha_\beta f\|^2_{L^2_{v,x}}\le \epsilon^2_1C_{\tau,T}<\infty,
%\end{align}
%where $C_{\tau,T}>0$ depends on $\tau,T$.
%Moreover, if additionally  
%\begin{align*}
%\sup_{l^*\ge 4}\E_{4,l^*}(0)
%\end{align*}is sufficiently small, then for any $l_0\ge K$, $|\alpha|+|\beta|\le K$, $k\ge 0$, $T<\infty$, we have 
%\begin{align}\label{19b}
%\sup_{\tau\le t\le T}\|w^{l_0-|\alpha|-|\beta|}\partial^\alpha_\beta\partial^{k}_tf\|^2_{L^2_{v,x}}\le C_{\tau,T,k}<\infty,
%\end{align}where $C_{\tau,T,k}$ is a constant depending on $\tau$, $T$, $k$. 
%Consequently, $f\in C^\infty(\R^+_t;C^\infty(\R^3_x;\mathscr{S}(\R^3_v)))$. 

\end{Thm}

Notice that \eqref{19a} gives the smoothing effect on velocity and spatial variable. If we assume the initial data has more velocity decay, then we have the smoothing effect on time variable as \eqref{19b}. If we assume the initial data as in the existence theory (cf. Theorem \ref{thm21}), then the constants can be independent of time $T$. Moreover, if we assume higher velocity decay, then we can derive \eqref{19aabb} from existence theory instead of assuming it at the beginning.  
These results show that the solutions to the Vlasov-Poisson-Boltzmann system enjoy a similar smoothing effect to the Boltzmann equation; see \cite{Alexandre2010,Chen2012}. That is, whenever the initial data has algebraic decay in any order, the solution $f$ is smooth in $(t,x,v)$ for any positive time $t$.

\smallskip
In what follows let us point out several technical points in the proof of Theorem \ref{main2}. 
We use $K\ge 4$ because $H^2_x(\R^3)$ is a Banach algebra when controlling \eqref{12}, where there has already second derivatives on $v$, and $H^2_x$ is useful to control the spatial variable when dealing with the trilinear estimate. 
The next technical point concerns the choice of $\psi=t^N$ in Theorem \ref{main2} and the usage of $\tilde{b}$, $\psi_{|\alpha|+|\beta|-4}$ is Section \ref{sec5}. Recall \eqref{psi} for definition of $\psi_k$. 
%Firstly, denote \begin{equation*}
	%\psi_k=\left\{\begin{aligned}
		%1, \text{  if $k\le 0$},\\
		%\psi^k, \text{ if $k> 0$}. 
		%\end{aligned}\right.
	%\end{equation*} 
Whenever $|\alpha|+|\beta|> 4$, $\psi_{|\alpha|+|\beta|-4} = t^{N(|\alpha|+|\beta|-4)}$ is equal to $0$ at $t=0$. Plugging this into energy estimate, the higher order derivatives are canceled at $t=0$ and one can control the higher order instant energy by lower order initial data. Then one can easily deduce the smoothing effect locally in time. By using the global energy control obtained in Theorem \ref{thm21}, the local-in-time regularity becomes global-in-time regularity. Notice that we use $-4$ to eliminate the index arising from Sobolev embedding $\|\cdot\|_{H^2_vL^\infty_x}\lesssim \|\cdot\|_{H^2_vH^2_x}$, where the latter has derivatives of forth order. However, after adding $\psi_{|\alpha|+|\beta|-4}$, one need to deal with the term 
\begin{align}\label{22}
	\big(\partial_t(\psi_{|\alpha|+|\beta|-4})\partial^\alpha_\beta f,e^{\pm\phi}w^2_l(\al,\beta)\partial^\alpha_\beta f\big)_{L^2_{v,x}}.
\end{align}
This is where we need $\tilde{b}$ given in \eqref{b}. By choosing $N=N(\alpha)$ properly, one has interpolation 
\begin{align*}
	\psi_{|\alpha|-4-\frac{1}{2N}}
	&\lesssim \delta\,\tilde{b}^{1/2}+C_{0,\delta}\<v\>^{\frac{-l_0|\alpha|}{\delta_1}}|y|^{-|\alpha|}.
\end{align*} 
The first term can be absorbed while the second term eliminates $\alpha$ derivatives on $x$. Applying a similar interpolation on $v$ with $\tilde{a}$, we can control \eqref{22} by a high-order term and an algebraic decay term:
\begin{align*}
	\delta^2\|\psi_{|\alpha|+|\beta|-4}\tilde{b}^{1/2}w_l(\al,\beta)(\partial^\alpha_\beta f)^\wedge(v,y)\|^2_{L^2_{v,y}}+\delta^2\D_{K,l}+C_\delta\|\<v\>^{C_{K,l}}f\|^2_{L^2_{v,x}}.
\end{align*}
Defining $\theta$ by \eqref{107}, using the equation \eqref{7}-\eqref{9} and Poisson bracket $\{v\cdot y,\theta\}$, one can control the high-order term by using functional $\E_{K,l}$ and $\D_{K,l}$, where $\delta_1$ in $\tilde{b}$ should be chosen properly. Hence, we can obtain a closed energy estimate locally. 
Here, when dealing with soft potential, there occurs an algebraic decay term in $v$: $\|\<v\>^{C_{K,l}}f\|_{L^2_{v,x}}$ and we need to assume such norms are bounded initially, as observed in the Boltzmann equation; cf. \cite{Chen2012}. 
After obtaining a local regularity, we can combine it with the global energy control from existence theory; cf. \cite{Duan2013}. Then one can deduce the regularity globally in time.

\smallskip

The rest of the paper is arranged as follows. In Section \ref{Sec2}, we present some basic Lemmas for existence theory, estimate on $L, \Gamma$, and some tricks in energy estimates. In Section \ref{sec5}, we present the proof for regularity. 
%The Appendix is devoted to pseudo-differential calculus and Carleman representation.

\section{Preliminaries}\label{Sec2}

In this section, we list several basic lemmas corresponding to the existence theory of Vlasov-Poisson-Boltzmann system, linearized Boltzmann collision term $L_\pm$ and the bilinear Boltzmann collision operator $\Gamma_\pm$. 
The following Theorem comes from \cite[Theorem 1.1]{Duan2013}, except that we improve the index $K\ge 8$ to $K\ge 4$ and $1/2\le s<1$ to $0<s<1$. 

\begin{Thm}[\cite{Duan2013}, Theorem 1.1]
	\label{thm21}
	Let $-\frac{3}{2}-2s<\gamma\le -2s$, $0<s<1$, $K\ge 4$, $p\in(\frac{1}{2},1)$. Assume $l\ge 0$, $l>-\frac{3(\gamma+2s)}{4}$, $l_1=-\frac{5(\gamma+2s)}{4(1-p)}$ and  $f_0(x,v)=(f_{0,+}(x,v),f_{0,-}(x,v))$ satisfying $F_\pm(0,x,v)=\mu(v)+\sqrt{\mu(v)}f_{0,\pm}(x,v)\ge 0$. Assume $\psi=1$. 
	If 
	\begin{align}\label{20a}
		\epsilon_0 = (\E_{K,l+l_1}(0))^{1/2}+\|w_{l_2}f_0\|_{Z_1}+\|E_0\|_{L^1_x},
	\end{align}
	is sufficiently small, where $E_0(x)=E(0,x)$, $l_2>-\frac{5(\gamma+2s)}{4}$ is a constant. Then there exists a unique global solution $f(t,x,v)$ to the Cauchy problem \eqref{7}-\eqref{9} of the Vlasov-Poisson-Boltzmann system such that $F_\pm(t,x,v)=\mu(v)+(\mu(v))^{1/2}f_\pm(t,x,v)\ge 0$ and 
	\begin{equation}\label{15a}\begin{aligned}
			\E_{K,l+l_1}(t)&\lesssim \epsilon_0^2,\\
			\E_{K,l}(t)&\lesssim \epsilon_0^2(1+t)^{-\frac{3}{2}},\\
			\E^h_{K,l}(t)&\lesssim \epsilon_0^2(1+t)^{-\frac{3}{2}-p},
		\end{aligned}
	\end{equation}
	for any $t\ge 0$. 
\end{Thm}
Here the instant energy functional $\E^h_{K,l}$ is given by 
\begin{align*}
	\E^h_{K,l}(t)\notag &\approx \sum_{|\alpha|\le K}\|\partial^\alpha E(t)\|^2_{L^2_x}+\sum_{1\le|\alpha|\le K}\|\partial^\alpha\P f\|^2_{L^2_{v,x}}\\
	&\qquad+\sum_{\substack{|\alpha|+|\beta|\le K}}\|w_l(\al,\beta)\partial^\alpha_\beta(\I-\P) f\|^2_{L^2_{v,x}},
\end{align*}
and we assume $\psi=1$ in this Theorem. 
\begin{proof}
	The proof is the similar to \cite[Theorem 1.1]{Duan2013} and we only illustrate the difference. The first one is that we use $\|E_0\|_{L^1_x}$ in \eqref{20a} instead of $\|(1+|x|)\rho_0\|_{L^1}$, where $\rho_0=\int_{\R^3}(f_+(0)-f_-(0))\mu^{1/2}\,dv$. The only place involving this term is estimate (4.25) in \cite[Theorem 1.1]{Duan2013}. One can use instead
	\begin{align*}
		\|\widehat{E_0}(y)\|_{L^\infty_y}\le \|E_0\|_{L^1_x},
	\end{align*}
and hence, in \eqref{20a}, we can use $\|E_0\|_{L^1_x}$ instead. 

\smallskip 
The second difference is that we use $K\ge 4$ instead of $K\ge 8$. 
%and $l\ge K$, $l>-\frac{3(\gamma+2s)}{4}+4$. 
This is because, in Corollary \ref{Coro1} below, we only require $K\ge 4$. Replacing estimate in \cite[Theorem 7.1, eq. (7.11)-(7.12)]{Duan2013} by Corollary \ref{Coro1} below, we can use such index on $K$ instead.

\smallskip 
The third difference is to improve index from $\frac{1}{2}\le s<1$ to $0<s<1$. The work \cite{Duan2013} is restricted to $\frac{1}{2}\le s<1$ because of \cite[Lemma 3.6 and 3.7]{Duan2013}, where the authors used Fourier transform on $v\in\R^3$ to control the gradient $\na_v$. Using Lemma \ref{Lem27} below instead, we are able to obtain the result for $0<s<1$. 
Then following the same proof of \cite[Lemma 7.1 and Theorem 1.1]{Duan2013}, we complete the proof of Theorem \ref{thm21}. 
\qe\end{proof}

Here we introduce the the following Lemmas from \cite{Deng2020a} on pseudo-differential calculus, which will be frequently used in our analysis. 
Notice that the condition $l\le m$ in \cite{Deng2020a} is unnecessary.
\begin{Lem}[\cite{Deng2020a}, Lemma 2.3]\label{inverse_bounded_lemma}Let $m,c$ be $\Gamma$-admissible weight and $a\in S(m)$.
	Assume $a^w:H(mc)\to H(c)$ is invertible.
	If $b\in S(m)$, then there exists $C>0$, depending only on the seminorms of symbols to $(a^w)^{-1}$ and $b^w$, such that for $f\in H(mc)$,
	\begin{align*}
		\|b(v,D_v)f\|_{H(c)}+\|b^w(v,D_v)f\|_{H(c)}\le C\|a^w(v,D_v)f\|_{H(c)}.
	\end{align*}
	Consequently, if $a^w:H(m_1)\to L^2\in Op(m_1)$, $b^w:H(m_2)\to L^2\in Op(m_2)$ are invertible, then for $f\in\mathscr{S}$, 
	\begin{align*}
		\|b^wa^wf\|_{L^2}\lesssim \|a^wb^wf\|_{L^2},
	\end{align*}where the constant depends only on seminorms of symbols to $a^w,b^w,(a^w)^{-1},(b^w)^{-1}$.
\end{Lem}
\begin{Lem}[\cite{Deng2020a}, Lemma 2.4]\label{bound_varepsilon}
	Denote $a_{K,l}:=a+Kl$, $m_{K,l}:=m+Kl$ for $K>1$, where $m,l$ are $\Gamma$-admissible weights. Assume $a\in S(m)$, $\partial_\eta (a_{K,l})\in S(K^{-\kappa}m_{K,l})$ uniformly in $K$ and
	$a_{K,l}\gtrsim m_{K,l}$.
	Let $\rho>0$ and $b\in S(\varepsilon m_{K,l}+\varepsilon^{-\rho}l)$, uniformly in $\varepsilon\in(0,1)$. Then there exists $K_0>0$, such that for $f\in H(mc)$, $\varepsilon\in(0,1)$,
	\begin{align*}
		\|b(v,D_v)f\|_{H(c)}+\|b^w(v,D_v)f\|_{H(c)}\le C_{K,l}\left(\varepsilon\|a^w(v,D_v)f\|_{H(c)}+\varepsilon^{-\rho}\|l^wf\|_{H(c)}\right).
	\end{align*}
\end{Lem}
For composition of pseudodifferential operator we have $a^wb^w= (a\#b)^w$ with 
\begin{align}\label{compostion}
	a\#b = ab + \frac{1}{4\pi i}\{a,b\} + \sum_{2\le k\le \nu}2^{-k}\sum_{|\alpha|+|\beta|=k}\frac{(-1)^{|\beta|}}{\alpha!\beta!}D^\alpha_\eta\partial^\beta_xaD^{\beta}_\eta\partial^\alpha_xb+r_\nu(a,b),
\end{align}where $X=(v,\eta)$,
\begin{align*}
	r_\nu(a,b)(X) & = R_\nu(a(X)\otimes b(Y))|_{X=Y},\\
	R_\nu &= \int^1_0\frac{(1-\theta)^{\nu-1}}{(\nu-1)!}\exp\Big(\frac{\theta}{4\pi i}\<\sigma\partial_X,\partial_Y\>\Big)\,d\theta\Big(\frac{1}{4\pi i}\<\sigma\partial_X,\partial_Y\>\Big)^\nu.
\end{align*}
Let $a_1(v,\eta)\in S(M_1,\Gamma),a_2(v,\eta)\in S(M_2,\Gamma)$, then $a_1^wa_2^w=(a_1\#a_2)^w$, $a_1\#a_2\in S(M_1M_2,\Gamma)$ with
\begin{align*}
	a_1\#a_2(v,\eta)&=a_1(v,\eta)a_2(v,\eta)
	+\int^1_0(\partial_{\eta}a_1\#_\theta \partial_{v} a_2-\partial_{v} a_1\#_\theta \partial_{\eta} a_2)\,d\theta,\\
	g\#_\theta h(Y):&=\frac{2^{2d}}{\theta^{-2n}}\int_\Rd\int_\Rd e^{-\frac{4\pi i}{\theta}\sigma(X-Y_1)\cdot(X-Y_2)}(4\pi i)^{-1}\<\sigma\partial_{Y_1}, \partial_{Y_2}\>g(Y_1) h(Y_2)\,dY_1dY_2,
\end{align*}with $Y=(v,\eta)$, $\sigma=\begin{pmatrix}
	0&I\\-I&0
\end{pmatrix}$.
For any non-negative integer $k$, there exists $l,C$ independent of $\theta\in[0,1]$ such that
\begin{align*}
	\|g\#_\theta h\|_{k;S(M_1M_2,\Gamma)}\le C\|g\|_{l,S(M_1,\Gamma)}\|h\|_{l,S(M_2,\Gamma)}.
\end{align*}
% with
% \begin{align}
%   \|g\|_{k;S(M,\Gamma)}:=\max_{0\le|\alpha|+|\beta|\le k}\sup_{v,\eta\in\R^d}\left|M(v,\eta)^{-1}\partial^\alpha_v\partial^\beta_\eta g(v,\eta)\right|.
% \end{align}
Thus if $\partial_{\eta}a_1,\partial_{\eta}a_2\in S(M'_1,\Gamma)$ and $\partial_{v}a_1,\partial_{v}a_2\in S(M'_2,\Gamma)$, then $[a_1,a_2]\in S(M'_1M'_2,\Gamma)$, where $[\cdot,\cdot]$ is the commutator defined by $[A,B]:=AB-BA$. 
As a consequence of composition and Lemma \ref{inverse_bounded_lemma}, we have the following.
\begin{Lem}\label{innerproduct}
	Let $m,c$ be $\Gamma$-admissible weight and $a^{1/2}\in S(m^{1/2})$.
	Assume $(a^{1/2})^w:H(mc)\to H(c)$ is invertible and $L\in S(m)$. Then 
	\begin{align*}
		(Lf,f)_{L^2} = (\underbrace{((a^{1/2})^w)^{-1}L}_{\in S(m^{1/2})}f,(a^{1/2})^wf)_{L^2}\lesssim \|(a^{1/2})^wf\|^2_{L^2}.
	\end{align*}
\end{Lem}

\medskip
The following lemma concerns with dissipation of $L_\pm$, whose proof can be found in \cite[Lemma 2.6 and Theorem 8.1]{Gressman2011}.  
\begin{Lem}\label{lemmaL}For any $l\in\R$, multi-indices $\alpha,\beta$, we have the followings. 
	
	(i) It holds that \begin{equation*}
		(-Lg,g)_{L^2_v}\gtrsim \|(\I-\P)g\|^2_{L^2_D}.
	\end{equation*}

(ii) There exists $C>0$ such that 
\begin{align*}
	-(w^2_{l}Lg,g)_{L^2_v}\gtrsim \|w_lg\|^2_{L^2_D}-C\|g\|^2_{L^2_v(B_C)}.
\end{align*}

(iii) For any $\eta>0$, 
\begin{multline*}
	-(w^2_l(\al,\beta)\partial^\alpha_\beta Lg,\partial^\alpha_\beta g)_{L^2_v}\gtrsim \|w_l(\al,\beta)\partial^\alpha_\beta g\|^2_{L^2_D}\\
	 - \eta\sum_{|\beta_1|\le|\beta|}\|w_l(\al,\beta_1)\partial^\alpha_{\beta_1}g\|^2_{L^2_D}-C_\eta\|\partial^\alpha g\|^2_{L^2(B_{C_\eta})}.
\end{multline*}

\end{Lem}

Notice that in Carleman representation (cf. \cite[Appendix]{Global2019}), the derivative on $v$ will apply to $f,g$ and $\mu^{1/2}$ respectively. Then, 
\begin{equation*}
\psi_{|\alpha|+|\beta|-4}\partial^\alpha_\beta\T(f,g) = \sum_{\alpha_1+\alpha_2=\alpha}\sum_{\beta_1+\beta_2+\beta_3=\beta}C^{\alpha_1,\alpha_2}_\alpha C^{\beta_1,\beta_2,\beta_3}_{\beta}\psi_{|\alpha|+|\beta|-4}\T_{\beta_3}(\partial^{\alpha_1}_{\beta_1}f,\partial^{\alpha_2}_{\beta_2}g)\psi_{|\beta_3|}. 
\end{equation*}
The next lemma concerns the estimates on the nonlinear collision operator $\Gamma_\pm$, which comes from \cite[Lemma 2.2]{Duan2013} and \cite[Proposition 3.1]{Strain2012}.
\begin{Lem}\label{Lem26a}
	Assume $\gamma+2s\le0$. For any $l\ge 0$, $m\ge 0$ and multi-index $\beta$, we have the upper bound 
	\begin{align}\label{12}
		|(w^2_l(\al,\beta)&\notag\partial^\alpha_\beta\Gamma_\pm(f,g),\partial^\alpha_\beta h)_{L^2_{v,x}}|\\&\notag\lesssim \sum_{\substack{\alpha_1+\alpha_2=\alpha\\\beta_1+\beta_2\le\beta}}\int_{\R^3}\|\partial^{\alpha_1}_{\beta_1}f\|_{L^2_v}\|w_l(\al,\beta)\partial^{\alpha_2}_{\beta_2}g\|_{L^2_D}\|w_l(\al,\beta)\partial^\alpha_\beta h\|_{L^2_D}\,dx\\&+ \sum_{\substack{\alpha_1+\alpha_2=\alpha\\\beta_1+\beta_2\le\beta}}\int_{\R^3}\|w_l(\al,\beta)\partial^{\alpha_1}_{\beta_1}f\|_{L^2_v}\|\partial^{\alpha_2}_{\beta_2}g\|_{L^2_D}\|w_l(\al,\beta)\partial^\alpha_\beta h\|_{L^2_D}\,dx\\&\notag+ \sum_{\substack{\alpha_1+\alpha_2=\alpha\\\beta_1+\beta_2\le\beta}}\int_{\R^3}\min\Big\{\sum_{|\beta'|\le2}\|w^{-m}\partial^{\alpha_1}_{\beta_1+\beta'}f\|_{L^2_v}\|w_l(\al,\beta)\partial^{\alpha_2}_{\beta_2}g\|_{L^2_D},\\&\notag\qquad\qquad\qquad\|w^{-m}\partial^{\alpha_1}_{\beta_1}f\|_{L^2_v}\sum_{|\beta'|\le2}\|w_l(\al,\beta)\partial^{\alpha_2}_{\beta_2+\beta'}g\|_{L^2_D}\Big\}\|w_l(\al,\beta)\partial^\alpha_\beta h\|_{L^2_D}\,dx.
	\end{align}
Let $i=1$ if $0<s<1/2$ and $i=2$ if $1/2\le s<1$, then 
\begin{align}\label{12a}
	\|\<v\>^l\Gamma(f,g)\|_{L^2_v}\lesssim \min\big\{\|\<v\>^{l+\frac{\gamma+2s}{2}}f\|_{H^2_v}\|\<v\>^{l+\frac{\gamma+2s}{2}}g\|_{H^i_v},\|\<v\>^{l+\frac{\gamma+2s}{2}}f\|_{L^2_v}\|\<v\>^{l+\frac{\gamma+2s}{2}}g\|_{H^{i+2}_v}\big\}.
\end{align}
\end{Lem}

In order to obtain a suitable norm estimate of $\T$ on $x$. We write a fundamental estimate, which is very useful throughout our analysis. 

\begin{Lem}\label{Lem27a}
	For any $u,v\in H^2_x$, we have 
	\begin{align}\label{13}
			\|uv\|_{L^2_x}&\lesssim \min\{\|\nabla_xu\|_{H^1_x}\|v\|_{L^2_x}, \|\nabla_xu\|_{L^2_x}\|v\|_{H^1_x}\}.
	\end{align}
%Consequently, 
%	\begin{align}\label{14}
%		\|uv\|_{H^2_x}\le \|\nabla_xu\|_{H^1_x}\|\nabla_xv\|_{H^1_x}. 
%	\end{align}
\end{Lem}
\begin{proof}
	The proof is straightforward. Notice that this lemma give that $H^2_x$ is a Banach algebra. 
	By Gagliardo–Nirenberg interpolation inequality and  Sobolev embedding; cf. \cite[Theorem 12.83]{Leoni2017} and \cite[Proposition 2.2 and Lemma 5.1]{Strain2013}, we have 
	\begin{align*}
		\|u\|_{L^\infty}&\lesssim \|\nabla_xu\|^{1/2}\|\nabla^2_xu\|^{1/2}\lesssim \|\nabla_xu\|_{H^1},\\
		\|uv\|_{L^2}&\lesssim \|u\|_{L^6}\|v\|_{L^3}\lesssim \|\nabla_xu\|_{L^2}\|v\|_{H^1}.
	\end{align*}
Then \eqref{13} follows from H\"{o}lder's inequality. 
%For \eqref{14}, 
%	\begin{align*}
%		\|uv\|_{H^2_x}\notag &= \sum_{|\alpha|\le 2}\|\partial^\alpha(uv)\|_{L^2}\\
%		&\lesssim \sum_{|\alpha|=2}\|u\partial^\alpha v\|_{L^2}+\sum_{|\alpha|=|\beta|=1}\|\partial^\alpha u\partial^\beta v\|_{L^2} + \sum_{|\alpha|=2}\|\partial^\alpha uv\|_{L^2}\\
%		&\lesssim \notag\|u\|_{L^\infty}\|\nabla_xv\|_{H^1}+\sum_{|\alpha|=|\beta|=1}\|\partial^\alpha u\|_{L^3}\|\partial^\beta v\|_{L^6} + \|\nabla_xu\|_{H^1}\|v\|_{L^\infty}.
%	\end{align*}
%Plugging the \eqref{13} estimate into this inequality, we have the desired control. 
\qe\end{proof}

The following Corollary gives the behavior of nonlinear terms in Vlasov-Poisson-Boltzmann system. 
\begin{Coro}\label{Coro1} Let $l\ge 0$ and $K\ge 4$. Define $i=1$ if $0<s<\frac{1}{2}$ and $i=2$ if $\frac{1}{2}\le s<1$. Assume $l>\max\{-\frac{3(\gamma+2s)}{4}+i+1, -\frac{5(\gamma+2s)}{4}+2\}$. Then, there exists  $l_*>-\frac{5(\gamma+2s)}{4}$ such that 
	\begin{align*}
	\|\<v\>^{l_*}g_\pm\|_{Z_1}+\|\<v\>^{l_*}\nabla_xg_\pm\|_{L^2_{v,x}}\lesssim \E_{K,l},
	\end{align*}where $g_\pm = \pm\nabla_x\phi\cdot\nabla_vf_\pm\mp\frac{1}{2}\nabla_x\phi\cdot vf_\pm+\Gamma_\pm(f,f)$. 
\end{Coro}
\begin{proof}
	By using \eqref{12a} and Young's inequality, we have 
	\begin{align*}
		\|\<v\>^{l_*}\Gamma(f,f)\|_{Z_1}&\lesssim \int dx\,\|\<v\>^{l_*+\gamma/2+s}f\|_{H^2_v}\|\<v\>^{l_*+\gamma/2+s}f\|_{H^i_v}\\
		&\lesssim \|\<v\>^{l_*+\frac{\gamma+2s}{2}}f\|^2_{H^2_vL^2_x}
		\lesssim \E_{K,l},
	\end{align*}  whenever $l\ge l_*+\frac{\gamma+2s}{2}+2$.  
	On the other hand,
	\begin{align*}
		\|\<v\>^{l_*}\nabla_x\phi\cdot\nabla_vf_\pm\|_{Z_1}&\lesssim \|\nabla_x\phi\|_{L^2_x}\|\<v\>^{l_*}\nabla_vf\|_{L^2_{v,x}}\lesssim\E_{K,l},\\
		\|\<v\>^{l_*}v\cdot\nabla_x\phi f_\pm\|_{Z_1}&\lesssim\|\nabla_x\phi\|_{L^2_x}\|\<v\>^{l_*}vf_\pm\|_{L^2_{v,x}}\lesssim\E_{K,l},
	\end{align*}whenever $l\ge l_*+1$. 
	Similarly, by using \eqref{13}, 
	\begin{align*}\notag
		\|\<v\>^{l_*}\nabla_x\Gamma(f,f)\|_{L^2_{v,x}}&\lesssim \Big\|\|\<v\>^{l_*+\frac{\gamma+2s}{2}}\nabla_xf\|_{L^2_v}\|\<v\>^{l_*+\frac{\gamma+2s}{2}}f\|_{H^i_v}\Big\|_{L^2_x}\\&\qquad+\Big\|\|\|\<v\>^{l_*+\frac{\gamma+2s}{2}}f\|_{L^2_v}\|\<v\>^{l_*+\frac{\gamma+2s}{2}}\nabla_xf\|_{H^i_v}\Big\|_{L^2_x}\\
		&\lesssim \|\<v\>^{l_*+\frac{\gamma+2s}{2}}f\|_{L^2_{v}H^2_{x}}\|\<v\>^{l_*+\frac{\gamma+2s}{2}}f\|_{H^i_vH^1_x}\\
		&\lesssim \E_{K,l},
	\end{align*}whenever $l\ge l_*+\frac{\gamma+2s}{2}+i+1$. By \eqref{13}, 
	\begin{align*}
		\|\<v\>^{l_*}\nabla_x(\nabla_x\phi\cdot\nabla_vf_\pm)\|_{L^2_{v,x}}&\lesssim \|\nabla_x\phi\|_{H^2_x}\|\<v\>^{l_*}f_\pm\|_{H^1_{v}H^1_x}\lesssim\E_{K,l}\\
		\|\<v\>^{l_*}\nabla_x(v\cdot\nabla_x\phi f_\pm)\|_{L^2_{v,x}}&\lesssim \|\nabla_x\phi\|_{H^1_x}\|\<v\>^{l_*}vf_\pm\|_{L^2_{v}H^1_x}\lesssim\E_{K,l},
	\end{align*}whenever $l\ge l_*+2$. 
	Now we verify that such $l_*$ exists. From the restriction above, we need to choose $l_*$ such that 
	\begin{align*}
		-\frac{5(\gamma+2s)}{4}< l_* \le l-\frac{\gamma+2s}{2}-i-1,\quad l_* \le l-2.
	\end{align*}Such choice exists, since $l> \max\{-\frac{3(\gamma+2s)}{4}+i+1, -\frac{5(\gamma+2s)}{4}+2\}$.
	
\qe\end{proof}

With the help of Lemma \ref{Lem26a} and \ref{Lem27a}, we can control the trilinear term as the following.
% $(\partial^\alpha_\beta\Gamma_\pm(f,g),w^2_l(\al,\beta)\partial^\alpha_\beta h)_{L^2_{v,x}}$. 
\begin{Lem}\label{lemmat}
	Let $K\ge 4$. For any multi-indices $|\alpha|+|\beta|\le K$ and real number $l\ge 0$, we have \begin{equation*}\begin{aligned}
		\Big|&(\psi_{2|\alpha|+2|\beta|-8}w^2_l(\al,\beta)\notag\partial^\alpha_\beta\Gamma_\pm(f,g),\partial^\alpha_\beta h)_{L^2_{v,x}}\Big|\\&\lesssim \bigg(\sum_{|\alpha|+|\beta|\le K}\|\psi_{|\alpha|+|\beta|-4}\partial^{\alpha}_\beta f\|_{L^2_{v,x}}\sum_{\substack{|\alpha|\ge 1\\ |\alpha|+|\beta|\le K}}\|\psi_{|\alpha|+|\beta|-4}w_l(\al,\beta)\partial^{\alpha}_{\beta}g\|_{L^2_xL^2_D}\\
		&\quad+\sum_{\substack{|\alpha|\ge 1\\ |\alpha|+|\beta|\le K}}\|\psi_{|\alpha|+|\beta|-4}\partial^{\alpha}_\beta f\|_{L^2_{v,x}}
		\sum_{|\alpha|+|\beta|\le K}\|\psi_{|\alpha|+|\beta|-4}w_l(\al,\beta)\partial^{\alpha}_{\beta}g\|_{L^2_xL^2_D}\\
		&\quad+\sum_{|\alpha|+|\beta|\le K}\|\psi_{|\alpha|+|\beta|-4}w_l(\al,\beta)\partial^{\alpha}_\beta f\|_{L^2_{v,x}}\sum_{\substack{|\alpha|\ge 1\\ |\alpha|+|\beta|\le K}}\|\psi_{|\alpha|+|\beta|-4}\partial^{\alpha}_{\beta}g\|_{L^2_xL^2_D}\\
		&\quad+\sum_{\substack{|\alpha|\ge 1\\ |\alpha|+|\beta|\le K}}\|\psi_{|\alpha|+|\beta|-4}w_l(\al,\beta)\partial^{\alpha}_\beta f\|_{L^2_{v,x}}
		\sum_{|\alpha|+|\beta|\le K}\|\psi_{|\alpha|+|\beta|-4}\partial^{\alpha}_{\beta}g\|_{L^2_xL^2_D}\bigg)
		\\
		&\qquad\qquad\qquad\qquad\qquad\qquad\qquad\qquad\qquad\times\|	\psi_{|\alpha|+|\beta|-4}w_l(\al,\beta)\partial^{\alpha}_\beta h\|_{L^2_xL^2_D},\end{aligned}
	\end{equation*}
	where we restrict $t\in[0,1]$ when considering $\psi=t^N$ as in Theorem \ref{main2}.
\end{Lem}
\begin{proof}
Using the estimate \eqref{12}, we have 
\begin{align}\label{27}\notag
&\notag\quad\,\big|(\psi_{2|\alpha|+2|\beta|-8}w^2_l(\al,\beta)\partial^\alpha_\beta\Gamma_\pm(f,g),\partial^\alpha_\beta h)_{L^2_{v,x}}\big| \\
&\notag\lesssim \sum_{\substack{\alpha_1+\alpha_2=\alpha\\\beta_1+\beta_2\le\beta}}\Big\|\psi_{|\alpha|+|\beta|-4}\|\partial^{\alpha_1}_{\beta_1}f\|_{L^2_v}\|w_l(\al,\beta)\partial^{\alpha_2}_{\beta_2}g\|_{L^2_D}\Big\|_{L^2_x}\\
&\notag\qquad\qquad\qquad\qquad\qquad\qquad\qquad\qquad\times\|\psi_{|\alpha|+|\beta|-4}w_l(\al,\beta)\partial^\alpha_\beta h\|_{L^2_xL^2_D}\\&\notag+ \sum_{\substack{\alpha_1+\alpha_2=\alpha\\\beta_1+\beta_2\le\beta}}\Big\|\psi_{|\alpha|+|\beta|-4}\|w_l(\al,\beta)\partial^{\alpha_1}_{\beta_1}f\|_{L^2_v}\|\partial^{\alpha_2}_{\beta_2}g\|_{L^2_D}\Big\|_{L^2_x}\\
&\notag\qquad\qquad\qquad\qquad\qquad\qquad\qquad\qquad\times\|\psi_{|\alpha|+|\beta|-4}w_l(\al,\beta)\partial^\alpha_\beta h\|_{L^2_xL^2_D}\\&\notag+ \sum_{\substack{\alpha_1+\alpha_2=\alpha\\\beta_1+\beta_2\le\beta}}\Big\|\psi_{|\alpha|+|\beta|-4}\min\Big\{\sum_{|\beta'|\le2}\|w^{-m}\partial^{\alpha_1}_{\beta_1+\beta'}f\|_{L^2_v}\|w_l(\al,\beta)\partial^{\alpha_2}_{\beta_2}g\|_{L^2_D},\\
&\notag\qquad\qquad\qquad\qquad\qquad\qquad\|w^{-m}\partial^{\alpha_1}_{\beta_1}f\|_{L^2_v}\sum_{|\beta'|\le2}\|w_l(\al,\beta)\partial^{\alpha_2}_{\beta_2+\beta'}g\|_{L^2_D}\Big\}\Big\|_{L^2_x}\\
&\qquad\qquad\qquad\qquad\qquad\qquad\qquad\qquad\times\|\psi_{|\alpha|+|\beta|-4}w_l(\al,\beta)\partial^\alpha_\beta h\|_{L^2_xL^2_D}.
\end{align}
%by replacing $l$ by $l-|\alpha|-|\beta|$ in \eqref{12}. 
Here we divide the summation into several parts. For brevity we denote the first terms in the norm $\|\cdot\|_{L^2_x}$ inside the summation $\sum_{\substack{\alpha_1+\alpha_2=\alpha\\\beta_1+\beta_2\le\beta}}$ on the right hand side of \eqref{27} to be $I,J,K$ and discuss their value in several cases. 
If $2\le|\alpha_1|+|\beta_1|\le K$, then $|\alpha_2|+|\beta_2|\le |\alpha|+|\beta|-2$ and $|\alpha_2+\alpha'|+|\beta_2|\le|\alpha|+|\beta|$ for any $1\le|\alpha'|\le2$. Notice that in this case, $\psi_{|\alpha|+|\beta|-4}\le \psi_{|\alpha_1|+|\beta_1|-4}\psi_{|\alpha_2+\alpha'|+|\beta_2|-4}$. By using \eqref{13}, we have 
\begin{align}\label{33a}
	\notag I&\lesssim\psi_{|\alpha|+|\beta|-4}\|\partial^{\alpha_1}_{\beta_1}f\|_{L^2_{v,x}}\big\|\|w_l(\al,\beta)\partial^{\alpha_2}_{\beta_2}g\|_{L^2_D}\big\|_{L^\infty_x}\\
	\notag&\lesssim\|\psi_{|\alpha_1|+|\beta_1|-4}\partial^{\alpha_1}_{\beta_1}f\|_{L^2_{v,x}}\sum_{1\le|\alpha'|\le2}\|\psi_{|\alpha_2+\alpha'|+|\beta_2|-4}w_l(\al+\al',\beta_2)\partial^{\alpha_2+\alpha'}_{\beta_2}g\|_{L^2_xL^2_D}\\
	&\lesssim\sum_{|\alpha|+|\beta|\le K}\|\psi_{|\alpha|+|\beta|-4}\partial^{\alpha}_\beta f\|_{L^2_{v,x}}\sum_{\substack{|\alpha|\ge 1\\ |\alpha|+|\beta|\le K}}\|\psi_{|\alpha|+|\beta|-4}w_l(\al,\beta)\partial^{\alpha}_{\beta}g\|_{L^2_xL^2_D}.
\end{align}
Secondly, if $|\alpha_1|+|\beta_1|=1$, then $|\alpha_2|+|\beta_2|\le |\alpha|+|\beta|-1$. Using \eqref{13} to give one $x$ derivative to $f$, we have 
\begin{align*}
	I&\lesssim \sum_{|\alpha'|= 1}\|\psi_{|\alpha_1+\alpha'|+|\beta_1|-4}\partial^{\alpha_1+\alpha'}_{\beta_1}f\|_{L^2_{v,x}}
	\\&\qquad\qquad\qquad\times\sum_{|\alpha'|\le 1}\|\psi_{|\alpha_2+\alpha'|+|\beta_2|-4}w_l(\al+\al',\beta_2)\partial^{\alpha_2+\alpha'}_{\beta_2}g\|_{L^2_xL^2_D}\\
	&\lesssim \sum_{\substack{|\alpha|\ge 1\\|\alpha|+|\beta|\le K}}\|\psi_{|\alpha|+|\beta|-4}\partial^{\alpha}_\beta f\|_{L^2_{v,x}}\sum_{|\alpha|+|\beta|\le K}\|\psi_{|\alpha|+|\beta|-4}w_l(\al,\beta)\partial^{\alpha}_{\beta}g\|_{L^2_xL^2_D}.
 \end{align*}
Here we used $\psi\le1$ and $\psi_{|\alpha|+|\beta|-4}\le \psi_{|\alpha_1+\alpha'_1|+|\beta_1|-4}\psi_{|\alpha_2+\alpha'_2|+|\beta_2|-4}$, for any $|\alpha'_1|= 1$, $|\alpha'_2|\le 1$.  
Thirdly, if $|\alpha_1|+|\beta_1|=0$, using \eqref{13} to give at most two and at least one spatial derivatives to $f$ with, we have 
\begin{align}\label{33b}
	I\notag&\lesssim\sum_{1\le|\alpha'|\le2}\|\psi_{|\alpha_1+\alpha'|+|\beta_1|-4}\partial^{\alpha_1+\alpha'}_{\beta_1}f\|_{L^2_{v,x}}\|\psi_{|\alpha_2|+|\beta_2|-4}w_l(\al_2,\beta_2)\partial^{\alpha_2}_{\beta_2}g\|_{L^2_xL^2_D}\\
	&\lesssim \sum_{\substack{|\alpha|\ge 1\\ |\alpha|+|\beta|\le K}}\|\psi_{|\alpha|+|\beta|-4}\partial^{\alpha}_\beta f\|_{L^2_{v,x}}
	\sum_{|\alpha|+|\beta|\le K}\|\psi_{|\alpha|+|\beta|-4}w_l(\al,\beta)\partial^{\alpha}_{\beta}g\|_{L^2_xL^2_D}.
\end{align}
Here we used $\psi_{|\alpha|+|\beta|-4}\le \psi_{|\alpha_1+\alpha'|+|\beta_1|-4}\psi_{|\alpha_2|+|\beta_2|-4}$, for any $|\alpha'|\le2$.  
Combining the above estimate, we have the desired result for $I$:
\begin{align*}
	I&\lesssim \sum_{|\alpha|+|\beta|\le K}\|\psi_{|\alpha|+|\beta|-4}\partial^{\alpha}_\beta f\|_{L^2_{v,x}}\sum_{\substack{|\alpha|\ge 1\\ |\alpha|+|\beta|\le K}}\|\psi_{|\alpha|+|\beta|-4}w_l(\al,\beta)\partial^{\alpha}_{\beta}g\|_{L^2_xL^2_D}\\
	&\qquad+\sum_{\substack{|\alpha|\ge 1\\ |\alpha|+|\beta|\le K}}\|\psi_{|\alpha|+|\beta|-4}\partial^{\alpha}_\beta f\|_{L^2_{v,x}}
	\sum_{|\alpha|+|\beta|\le K}\|\psi_{|\alpha|+|\beta|-4}w_l(\al,\beta)\partial^{\alpha}_{\beta}g\|_{L^2_xL^2_D}.
\end{align*}
Similarly, using the same discussion on $|\alpha_2|+|\beta_2|$ instead of $|\alpha_1|+|\beta_1|$, we have 
\begin{align*}
	J&\lesssim \sum_{|\alpha|+|\beta|\le K}\|\psi_{|\alpha|+|\beta|-4}w_l(\al,\beta)\partial^{\alpha}_\beta f\|_{L^2_{v,x}}\sum_{\substack{|\alpha|\ge 1\\ |\alpha|+|\beta|\le K}}\|\psi_{|\alpha|+|\beta|-4}\partial^{\alpha}_{\beta}g\|_{L^2_xL^2_D}\\
	&\qquad+\sum_{\substack{|\alpha|\ge 1\\ |\alpha|+|\beta|\le K}}\|\psi_{|\alpha|+|\beta|-4}w_l(\al,\beta)\partial^{\alpha}_\beta f\|_{L^2_{v,x}}
	\sum_{|\alpha|+|\beta|\le K}\|\psi_{|\alpha|+|\beta|-4}\partial^{\alpha}_{\beta}g\|_{L^2_xL^2_D}.
\end{align*}
For the term $K$, the idea is similar to $I$. If $|\alpha_1|+|\beta_1|=0$, we use the first term in minimum of $K$ and apply \eqref{13} to give at most two and at least one spatial derivatives to $f$. Noticing $\psi_{|\alpha|+|\beta|-4}\le \psi_{|\alpha_1+\alpha'|+|\beta_1+\beta'|-4}\psi_{|\alpha_2|+|\beta_2|-4}$, for $1\le|\alpha'|\le 2,|\beta'|\le2$, we have 
\begin{align*}
	K&\lesssim \psi_{|\alpha|+|\beta|-4}\sum_{1\le|\alpha'|\le 2,|\beta'|\le2}\|w^{-m}\partial^{\alpha_1+\alpha'}_{\beta_1+\beta'}f\|_{L^2_{v,x}}\|w_l(\al,\beta)\partial^{\alpha_2}_{\beta_2}g\|_{L^2_xL^2_D}\\
	&\lesssim \sum_{\substack{|\alpha|\ge 1\\ |\alpha|+|\beta|\le K}}\|\psi_{|\alpha|+|\beta|-4}\partial^{\alpha}_{\beta}f\|_{L^2_{v,x}} \sum_{\substack{|\alpha|+|\beta|\le K}}\|\psi_{|\alpha|+|\beta|-4}w_l(\al,\beta)\partial^{\alpha}_{\beta}g\|_{L^2_xL^2_D}.
\end{align*}
Simlarly, if $|\alpha_1|+|\beta_1|=1$, we apply \eqref{13} to give at least one $x$ derivative to $f$, at most one $x$ derivative to $g$ and deduce the same bound. If $|\alpha_1|+|\beta_1|=2$, we apply \eqref{13} to give at most two and at least one spatial derivatives to $g$. Noticing $\psi_{|\alpha|+|\beta|-4}\le \psi_{|\alpha_1|+|\beta_1+\beta'|-4}\psi_{|\alpha_2+\alpha'|+|\beta_2|-4}$, for $1\le|\alpha'|\le 2,|\beta'|\le2$, we have 
\begin{align*}
K&\lesssim \psi_{|\alpha|+|\beta|-4}\sum_{|\beta'|\le2}\|w^{-m}\partial^{\alpha_1}_{\beta_1+\beta'}f\|_{L^2_{v,x}}\sum_{1\le|\alpha'|\le 2}\|w_l(\al,\beta)\partial^{\alpha_2+\alpha'}_{\beta_2}g\|_{L^2_xL^2_D}\\
&\lesssim \sum_{\substack{|\alpha|+|\beta|\le K}}\|\psi_{|\alpha|+|\beta|-4}\partial^{\alpha}_{\beta}f\|_{L^2_{v,x}} \sum_{\substack{|\alpha|\ge 1\\ |\alpha|+|\beta|\le K}}\|\psi_{|\alpha|+|\beta|-4}w_l(\al,\beta)\partial^{\alpha}_{\beta}g\|_{L^2_xL^2_D}.
\end{align*}
If $|\alpha_1|+|\beta_1|=3$, we will use the second term in the minimum of $K$. Applying \eqref{13} to give at least one $x$ derivative to $f$ and at most one $x$ derivative to $g$, noticing $\psi_{|\alpha|+|\beta|-4}\le \psi_{|\alpha_1+\alpha'_1|+|\beta_1|-4}\psi_{|\alpha_2+\alpha'_2|+|\beta_2+\beta'|-4}$ and $w_l(\al,\beta)\le w_l(\al_2+\al'_2,\beta_2+\beta')$ for any $|\alpha'_1|= 1,|\alpha'|\le 1,|\beta'|\le 2$, we have 
\begin{align*}
	K&\lesssim \psi_{|\alpha|+|\beta|-4}\sum_{|\alpha'_1|= 1}\|w^{-m}\partial^{\alpha_1+\alpha'_1}_{\beta_1}f\|_{L^2_{v,x}}\sum_{|\alpha'_2|\le 1, |\beta'|\le2}\|w_l(\al,\beta)\partial^{\alpha_2+\alpha'_2}_{\beta_2+\beta'}g\|_{L^2_xL^2_D}\\
	&\lesssim \sum_{\substack{|\alpha|\ge 1\\ |\alpha|+|\beta|\le K}}\|\psi_{|\alpha|+|\beta|-4}\partial^{\alpha}_{\beta}f\|_{L^2_{v,x}} \sum_{\substack{|\alpha|+|\beta|\le K}}\|\psi_{|\alpha|+|\beta|-4}w_l(\al,\beta)\partial^{\alpha}_{\beta}g\|_{L^2_xL^2_D}.
\end{align*}
 If $4\le|\alpha_1|+|\beta_1|\le K$, then applying \eqref{13} to give two $x$ derivatives to $g$ and noticing $\psi_{|\alpha|+|\beta|-4}\le \psi_{|\alpha_1|+|\beta_1|-4}\psi_{|\alpha_2+\alpha'|+|\beta_2+\beta'|-4}$ and $w_l(\al,\beta)\le w_l(\al_2+\al',\beta_2+\beta')$ for any $1\le|\alpha'|\le 2,|\beta'|\le2$, we have 
\begin{align*}
	K&\lesssim \psi_{|\alpha|+|\beta|-4}\|w^{-m}\partial^{\alpha_1}_{\beta_1}f\|_{L^2_{v,x}}\sum_{1\le|\alpha'|\le 2, |\beta'|\le2}\|w_l(\al,\beta)\partial^{\alpha_2+\alpha'}_{\beta_2+\beta'}g\|_{L^2_xL^2_D}\\
	&\lesssim \sum_{\substack{ |\alpha|+|\beta|\le K}}\|\psi_{|\alpha|+|\beta|-4}\partial^{\alpha}_{\beta}f\|_{L^2_{v,x}} \sum_{\substack{|\alpha|\ge 1\\|\alpha|+|\beta|\le K}}\|\psi_{|\alpha|+|\beta|-4}w_l(\al,\beta)\partial^{\alpha}_{\beta}g\|_{L^2_xL^2_D},
\end{align*}
Substituting all the above estimate into \eqref{27}, we have the desired bound. 
Similar discussion on the indices $|\alpha_1|+|\beta_1|$ will be used frequently later and will not be mentioned for brevity. 

%The proof of \eqref{15a} and \eqref{15b} are similar. But notice that when $\beta=0$, one can always produce one derivative on the first term $f$. 

\qe\end{proof}

A direct consequence of Lemma \ref{lemmat} is the following estimate; see also \cite[Lemma 3.1]{Duan2013}.
\begin{Lem}\label{lemmag}
	Let $K\ge 4$, $|\alpha|+|\beta|\le K$, $l\ge 0$. Then,
	\begin{equation}\label{210}
		|(\partial^\alpha\Gamma_\pm(f,f),\psi_{2|\alpha|-8}\partial^\alpha f_\pm)_{L^2_{v,x}}|\lesssim\E^{1/2}_{K,l}\D_{K,l}(t),
	\end{equation}
%If $|\alpha|\ge 1$, then \begin{equation*}
%	|(\partial^\alpha\Gamma_\pm(f,f),\partial^\alpha f_\pm)_{L^2_{v,x}}|\lesssim \E^{1/2}_{K,l,h}\D_{K,l}(t).
%\end{equation*}
and
%	\begin{equation*}
%		|(w^2_l(\al,\beta)\partial^\alpha_\beta\Gamma_\pm(f,f),\partial^\alpha_\beta(\II-\PP)f)_{L^2_{v,x}}|\lesssim \E^{1/2}_{K,l}\D_{K,l}(t),
%	\end{equation*}
%In particular, when $|\alpha|\ge 1$, 
%\begin{equation*}
%|(w^2_l(\al,\beta)\partial^\alpha_\beta\Gamma_\pm(f,f),\partial^\alpha_\beta f_\pm)_{L^2_{v,x}}|\lesssim \E^{1/2}_{K,l}\D_{K,l}(t).
%\end{equation*}
%When $0\le|\alpha|\le K$, 
\begin{equation}\label{25}
	|(w^2_l(\al,\beta)\partial^\alpha_\beta\Gamma_\pm(f,f),\psi_{2|\alpha|+2|\beta|-8}\partial^\alpha_\beta f)_{L^2_{v,x}}|\lesssim \E^{1/2}_{K,l}\D_{K,l}(t)+\E_{K,l}\D^{1/2}_{K,l}(t).
\end{equation}
Also, for any smooth function $\zeta(v)$ satisfying $|\zeta(v)|\approx e^{-\lambda|v|^2}$ with some $\lambda>0$, we have 
\begin{align}\label{212}
	(\partial^\alpha\Gamma_\pm(f,f),\psi_{2|\alpha|-8}\zeta(v))_{L_{v,x}}\lesssim \E^{1/2}_{K,l}\D^{1/2}_{K,l}(t).
\end{align}
\end{Lem}
\begin{proof}
%	For brevity, we only give the proof of \eqref{25}. 
For \eqref{25}, notice that 
	\begin{align*}
		&\quad\,(w^2_l(\al,\beta)\partial^\alpha_\beta\Gamma_\pm(f,f),\psi_{2|\alpha|+2|\beta|-8}\partial^\alpha_\beta f_\pm)_{L^2_{v,x}} \\
		&= (w^2_l(\al,\beta)\partial^\alpha_\beta\Gamma_\pm(f,f),\psi_{2|\alpha|+2|\beta|-8}\partial^\alpha_\beta (\II-\PP) f)_{L^2_{v,x}}\\
		&\qquad+(w^2_l(\al,\beta)\partial^\alpha_\beta\Gamma_\pm(f,f),\psi_{2|\alpha|+2|\beta|-8}\partial^\alpha_\beta \PP f)_{L^2_{v,x}}.
	\end{align*}
	The first term on the right hand, by directly using Lemma \ref{lemmat} and the definition of $\E_{K,l}$ and $\D_{K,l}$, is bounded above by $\E^{1/2}_{K,l}\D_{K,l}(t)$, since there's zero $x$ derivative on $(\I-\P)f$ in the definition of $\D_{K,l}$. But there's no such term for $\P f$ in $\D_{K,l}$. Hence, the second right-hand term can only be bounded above by $\E_{K,l}\D^{1/2}_{K,l}(t)$. This proves \eqref{25}. 
	
	Similarly, noticing $P_\pm\Gamma(f,f)=0$, one can obtain \eqref{210}. The proof of \eqref{212} is directly from Lemma \ref{lemmat}. This conclude Lemma \ref{lemmag}. 
%	One should note that the restriction $|\alpha|\ge 1$ and $|\alpha|\le K-1$ in \eqref{Defe}\eqref{Defd} as well as in Lemma \ref{lemmat} is crucial in our analysis. 
\qe\end{proof}

For later use, we need the following estimate on $v\cdot\nabla_x\phi f_\pm$ and $\nabla_x\phi\cdot\nabla_vf_\pm$. We always assume that $\|\phi\|_{L^\infty_x}\le C$, which follows from the {\it a priori} assumption on energy $\E_{K,l}$ given in \eqref{Defe} and hence, $|e^{\pm\phi}|\approx 1$.  
The proof here is different from \cite[Lemma 3.4 and 3.6]{Duan2013}, since we will cover the full range $0<s<1$. 
% Since we involve $\psi$, we write down their proofs explicitly. 
\begin{Lem}\label{Lem26}Let $1\le|\alpha|\le K$, $|\alpha|+|\beta|\le K$ and $l\ge 0$. Then, for $\alpha_1\le\alpha,\beta_1\le\beta$ with $|\alpha_1|\ge 1$, it holds that 
	\begin{equation*}
		|(v_i\partial^{\alpha_1+e_i}\phi\partial^{\alpha-\alpha_1}f_\pm,\psi_{2|\alpha|-8}e^{\pm\phi}w^2_l(\al,0)\partial^\alpha f_\pm)_{L^2_{v,x}}|\lesssim \E^{1/2}_{K,l}\D_{K,l}, 
	\end{equation*}
	\begin{equation*}
		|(\partial_{\beta_1}v_i\partial^{\alpha_1+e_i}\phi\partial^{\alpha-\alpha_1}_{\beta-\beta_1}f_\pm,\psi_{2|\alpha|+2|\beta|-8}e^{\pm\phi}w^2_l(\al,\beta)\partial^\alpha_\beta f_\pm)_{L^2_{v,x}}|\lesssim \E^{1/2}_{K,l}\D_{K,l}.
	\end{equation*}
%\begin{equation*}
%|(\partial_{\beta_1}v_i\partial^{\alpha_1+e_i}\phi\partial^{\alpha-\alpha_1}_{\beta-\beta_1}(\II-\PP)f,e^{\pm\phi}w^2_l(\al,\beta)\partial^\alpha_\beta (\II-\PP)f)_{L^2_{v,x}}|\lesssim \E^{1/2}_{K,l}\D_{K,l}.
%\end{equation*}
\end{Lem}
\begin{proof}
%	We will frequently use the following estimate for $1\le m\le K$. By \eqref{8},  
%	\begin{align}\label{47a}
%	\|\psi_{m+1-4}\nabla^{m+1}_x\nabla_x\phi\|_{L^2_x}\lesssim \|\psi_{m+1-4}\nabla^{m+1}_x\nabla_x\Delta^{-1}_x(a_+-a_-)\|_{L^2_x}\lesssim \sum_{\pm}\|\psi_{m-4}\nabla^{m}_xa_\pm\|_{L^2_{x}}.
%	\end{align}
For $|\alpha_1|\ge 1$ with $\alpha_1\le\alpha$, by using $-3<\gamma\le-2s$ and $0<s<1$, we have from \eqref{w} that  $|v_i|w_l(|\al|,0)\le \<v\>^{\gamma}w_l(|\al|-1,0)$. Thus, 
	\begin{align}\notag
	&\quad\,|(v_i\partial^{\alpha_1+e_i}\phi\partial^{\alpha-\alpha_1}f_\pm,\psi_{2|\alpha|-8}e^{\pm\phi}w^2_l(\al,0)\partial^\alpha f_\pm)_{L^2_{v,x}}|\\
	&\lesssim\label{26} \|\psi_{|\alpha|-4}\partial^{\alpha_1}\nabla_x\phi\<v\>^{\frac{\gamma}{2}}w_l(|\al|-1,0)\partial^{\alpha-\alpha_1}f_\pm\|_{L^2_{v,x}}\|\psi_{|\alpha|-4}\<v\>^{\frac{\gamma}{2}}w_l(|\al|,0)\partial^{\alpha}f_\pm\|_{L^2_{v,x}}. 
\end{align}
For the first term on the right hand of \eqref{26}, we discuss its value as the following. 
If $\alpha_1<\alpha$, then $1\le|\alpha_1|\le K-1$ and there's at least one derivative on $f_\pm$ with respect to $x$. Then by the same discussion on the value of $|\alpha_1|$ as \eqref{33a}-\eqref{33b}, one has 
\begin{align*}
%	\label{37bb}
	&\quad\,\|\psi_{|\alpha|-4}\partial^{\alpha_1}\nabla_x\phi\<v\>^{\frac{\gamma}{2}}w_l(|\al|-1,0)\partial^{\alpha-\alpha_1}f_\pm\|_{L^2_{v,x}}
	\lesssim \E^{1/2}_{K,l}\D^{1/2}_{K,l}, 
\end{align*}where we used $\|\<v\>^{\gamma/2+s}(\cdot)\|_{L^2_{v,x}}\lesssim \|\cdot\|_{L^2_xL^2_D}$. 
If $\alpha_1=\alpha$, then we decompose $f_\pm=\PP f+(\II-\PP)f$ and give one derivative to $\PP f$ with respect to $x$ by using \eqref{13}. 
That is, 
\begin{align*}
	&\quad\,\|\psi_{|\alpha|-4}\partial^{\alpha}\nabla_x\phi\<v\>^{\frac{\gamma}{2}}w_l(\al-\al_1,0)\PP f\|_{L^2_{v,x}}\notag\\
	&\lesssim\|\psi_{|\alpha|-4}\partial^\alpha\nabla_x\phi\|_{L^2_x}\sum_{1\le|\alpha'|\le 2}\|\psi_{|\alpha'|-4}\partial^{\alpha'}\PP f\|_{L^2_{v,x}}\\
	&\lesssim \E^{1/2}_{K,l}\D_{K,l}^{1/2}.
\end{align*}
For the part $(\II-\PP)f$, we will use \eqref{13} to give two derivatives to $(\II-\PP)f$ when $|\alpha|\ge 3$, one derivative to $(\II-\PP)f$ when $|\alpha|=2$ and give nothing to $(\II-\PP)f$ when $|\alpha|=1$. That is, 
\begin{align*}\notag
	&\quad\,\|\psi_{|\alpha|-4}\partial^{\alpha}\nabla_x\phi\<v\>^{\frac{\gamma}{2}}w_l(\al-\al_1,0)(\II-\PP)f\|_{L^2_{v,x}}\notag\\
	&\lesssim\notag \sum_{3\le|\alpha|\le K}\|\psi_{|\alpha|-4}\partial^{\alpha}\nabla_x\phi\|_{L^2_x}\sum_{1\le|\alpha'|\le2}\|\psi_{|\alpha'|-4}\<v\>^{\frac{\gamma}{2}}w_l(|\al|-1,0)\partial^{\alpha'}(\II-\PP)f\|_{L^2_{v,x}}\\\notag
	&\quad\notag+\sum_{|\alpha|=2}\sum_{|\alpha'|\le1}\|\psi_{|\alpha+\alpha'|-4}\partial^{\alpha+\alpha'}\nabla_x\phi\|_{L^2_x}\sum_{|\alpha_1'|=1}\|\psi_{|\alpha_1'|-4}\<v\>^{\frac{\gamma}{2}}w_l(|\al|-1,0)\partial^{\alpha_1'}(\II-\PP)f\|_{L^2_{v,x}}\\
	&\quad\notag+\sum_{|\alpha|=1}\sum_{|\alpha'|\le2}\|\psi_{|\alpha+\alpha'|-4}\partial^{\alpha+\alpha'}\nabla_x\phi\|_{L^2_x}\|\<v\>^{\frac{\gamma}{2}}w_{l}(\II-\PP)f\|_{L^2_{v,x}}\\
	&\lesssim \E^{1/2}_{K,l}\D^{1/2}_{K,l},
%	\label{29}
\end{align*}where we used $-4$ in $\psi$ through our argument.
Thus, when $\alpha_1=\alpha$, 
\begin{align}
	&\quad\,\|\psi_{|\alpha|-4}\partial^{\alpha_1}\nabla_x\phi\<v\>^{\frac{\gamma}{2}}w_l(|\al|-1,0)\partial^{\alpha-\alpha_1}f_\pm\|_{L^2_{v,x}}
	\lesssim \E^{1/2}_{K,l}\D^{1/2}_{K,l}.\label{27c}
	\end{align}
Plugging the above estimate into \eqref{26}, we have 
\begin{align*}
	|(v_i\partial^{\alpha_1+e_i}\phi\partial^{\alpha-\alpha_1}f_\pm,\psi_{2|\alpha|-8}e^{\pm\phi}w^2_l(\al,0)\partial^\alpha f_\pm)_{L^2_{v,x}}|\lesssim \E^{1/2}_{K,l}\D_{K,l}. 
\end{align*}
	
	Similarly, for $|\beta|\le K$ and $\beta_1\le \beta$, we have $|\partial_{\beta_1}v_i|\le \<v\>$ and hence, 
	\begin{align}\label{28}
	&\notag\quad\,|(\partial_{\beta_1}v_i\partial^{\alpha_1+e_i}\phi\partial^{\alpha-\alpha_1}_{\beta-\beta_1}f_\pm,\psi_{2|\alpha|+2|\beta|-8}e^{\pm\phi}w^2_l(\al,\beta)\partial^\alpha_\beta f_\pm)_{L^2_{v,x}}|\\
	&\notag\lesssim \|\psi_{|\alpha|+|\beta|-4}\partial^{\alpha_1}\nabla_x\phi\<v\>^{\frac{\gamma}{2}}w_l(|\alpha|-1,|\beta-\beta_1|)\partial^{\alpha-\alpha_1}_{\beta-\beta_1} f_\pm\|_{L^2_{v,x}}\\&\qquad\qquad\times\|\psi_{|\alpha|+|\beta|-4}\<v\>^{\frac{\gamma}{2}}w_l(\al,\beta)\partial^\alpha_\beta f_\pm\|_{L^2_{v,x}}.
\end{align}
For the first term on the right hand of \eqref{28}, we use the same argument as in \eqref{26}-\eqref{27c} to find its upper bound $\E^{1/2}_{K,l}\D^{1/2}_{K,l}$. Hence, \eqref{28} is bounded above by $\E^{1/2}_{K,l}\D_{K,l}$.

\qe\end{proof}

\begin{Lem}\label{Lem27}
	Let $|\alpha|+|\beta|\le K$, $l\ge 0$. Then, for $\alpha_1\le\alpha,\beta_1\le\beta$, it holds that 
	\begin{equation}\label{30a}
		|(\partial^{\alpha_1+e_i}\phi\partial^{\alpha-\alpha_1}_{e_i}f_\pm,\psi_{2|\alpha|-8}e^{\pm\phi}w^2_l(\al,0)\partial^\alpha f_\pm)_{L^2_{v,x}}|\le \E^{1/2}_{K,l}\D_{K,l}, 
	\end{equation}
and 
\begin{equation}\label{30b}
	|(\partial^{\alpha_1+e_i}\phi\partial^{\alpha-\alpha_1}_{\beta+e_i}f_\pm,\psi_{2|\alpha|+2|\beta|-8}e^{\pm\phi}w^2_l(\al,\beta)\partial^\alpha_\beta f_\pm)_{L^2_{v,x}}|\le \E^{1/2}_{K,l}\D_{K,l}.
\end{equation}
%\begin{equation}\label{30c}
%|(\partial^{\alpha_1+e_i}\phi\partial^{\alpha-\alpha_1}_{\beta+e_i}(\II-\PP)f,e^{\pm\phi}w^2_l(\al,\beta)\partial^\alpha_\beta (\II-\PP)f)_{L^2_{v,x}}|\le \E^{1/2}_{K,l}\D_{K,l}.
%\end{equation}
\end{Lem}
\begin{proof}We firstly prove \eqref{30a}.
	When $\alpha_1=0$, by integration by parts and $\gamma+2s\ge -2$, we have 
	\begin{align*}
	&\quad\,|(\partial^{e_i}\phi\partial^{\alpha}_{e_i}f_\pm,\psi_{2|\alpha|-8}e^{\pm\phi}w^2_l(\al,0)\partial^\alpha f_\pm)_{L^2_{v,x}}|\notag\\
	&\lesssim |(\partial^{e_i}\phi\partial^{\alpha}f_\pm,\psi_{2|\alpha|-8}e^{\pm\phi}(\partial_{e_i}w^2_l(\al,0))\partial^\alpha f_\pm)_{L^2_{v,x}}|\notag\\
	&\lesssim \|\psi_{|\alpha|-4}\nabla_x\phi \<v\>^{\frac{\gamma+2s}{2}}w_l(|\al|,0)\partial^{\alpha}f_\pm\|_{L^2_{v,x}}\|\psi_{|\alpha|-4}w_l(|\al|,0)\partial^\alpha f_\pm\|_{L^2_{v,x}}\notag\\
	&\lesssim \sum_{|\alpha'|\le 2}\|\psi_{|\alpha'|-4}\partial^{\alpha'}\nabla_x\phi\|_{L^2_x}\sum_{1\le|\alpha|\le K}\|\psi_{|\alpha|-4}\<v\>^{\frac{\gamma+2s}{2}}w_l(|\al|,0)\partial^\alpha f_\pm\|_{L^2_{v,x}}\notag\\&\qquad\qquad\qquad\qquad\qquad\qquad\qquad\times\sum_{|\alpha|\le K}\|\psi_{|\alpha|-4}w_l(|\al|,0)\partial^\alpha f_\pm\|_{L^2_{v,x}}\notag\\
	&\lesssim \E^{1/2}_{K,l}\D_{K,l}, 
%	\label{31a}
	\end{align*}where we use \eqref{13} to assure that there's always at least one derivative on the first $f_\pm$. 
	When $|\alpha_1|\ge 1$, we have $|\alpha|\ge 1$. Then we decompose $f_\pm=\PP f+(\II-\PP)f$ to obtain 
	\begin{align*}
	&\quad\,(\partial^{\alpha_1+e_i}\phi\partial^{\alpha-\alpha_1}_{e_i}f_\pm,\psi_{2|\alpha|-8}e^{\pm\phi}w^2_l(\al,0)\partial^\alpha f_\pm)_{L^2_{v,x}}=I+J,
\end{align*}with 
\begin{align*}
	I &= (\partial^{\alpha_1+e_i}\phi\partial^{\alpha-\alpha_1}_{e_i}\PP f,\psi_{2|\alpha|-8}e^{\pm\phi}w^2_l(\al,0)\partial^\alpha f_\pm)_{L^2_{v,x}},\\
	J&=(\partial^{\alpha_1+e_i}\phi\partial^{\alpha-\alpha_1}_{e_i}(\II-\PP)f,\psi_{2|\alpha|-8}e^{\pm\phi}w^2_l(\al,0)\partial^\alpha f_\pm)_{L^2_{v,x}}.
\end{align*}
Now we estimate $I$ and $J$ as the followings. For $I$, noticing there's exponential decay in $v$, we have 
\begin{align*}
	|I|&\lesssim\|\psi_{|\alpha|-4}\partial^{\alpha_1+e_i}\phi\partial^{\alpha-\alpha_1}\PP f\|_{L^2_{v,x}}\|\psi_{|\alpha|-4}\<v\>^{\frac{\gamma+2s}{2}}w_l(|\al|,0)\partial^\alpha f_\pm\|_{L^2_{v,x}}\\
	&\lesssim \sum_{|\alpha_1|\le K}\|\psi_{|\alpha_1|-4}\partial^{\alpha_1}\nabla_x\phi\|_{L^2_x}\sum_{1\le|\alpha|\le K}\|\psi_{|\alpha|-4}\partial^\alpha\PP f\|_{L^2_{v,x}}\|\psi_{|\alpha|-4}w_l(|\al|,0)\partial^\alpha f_\pm\|_{L^2_xL^2_D}\\
	&\lesssim \E^{1/2}_{K,l}\D_{K,l},
\end{align*}where we used same discussion on the value of $|\alpha_1|$ as \eqref{33a}-\eqref{33b} and give at least one derivative to $\PP f$. 
For $J$, we first provide some interpolation formulas. 
For any $k\in\R$, by Young's inequality, we have $\<\eta\>\lesssim \<\eta\>^s\<v\>^k + \<\eta\>^{1+s}\<v\>^{-\frac{ks}{1-s}}$ and hence, $\<\eta\>$ is a symbol in $S(\<\eta\>^s\<v\>^k + \<\eta\>^{1+s}\<v\>^{-\frac{ks}{1-s}})$, where $\eta$ is the Fourier variable of $v$. Then by \cite[Lemma 2.3 and Corollary 2.5]{Deng2020a}, we have 
	\begin{equation}\label{fff}
			\|f\|_{H^1_v}
			%	 \lesssim\| f \langle v \rangle^{k} \|_{H^s}^{s} \| f \langle v \rangle^{ - k s /(1-s)}\|_{H^{1+s}}^{1-s} 
			\lesssim \| f \<v\>^{k} \|_{H^s} +\| f \<v\>^{-ks/(1-s)}\|_{H^{1+s}}.
		\end{equation}
%	 Applying Gagliardo–Nirenberg interpolation inequality in \cite[Theorem 12.83]{Leoni2017}, we obtain 
%		\begin{multline}\label{Linfty}
%		\|f\|_{L^\infty_xL^2_v}
%		\le \Big(\int_{\R^3}\|f\|_{L^\infty_x}^2\,dv\Big)^{\frac{1}{2}}
%		\lesssim \Big(\int_{\R^3}\<v\>^{k}\|\na_xf\|_{L^2_x}\<v\>^{-k}\|\na^2_xf\|_{L^2_x}\,dv\Big)^{\frac{1}{2}}
%		\\
%		\lesssim \|\<v\>^{k}\na_xf\|^{\frac{1}{2}}_{L^2_xL^2_v}\|\<v\>^{-k}\na^2_xf\|_{L^2_xL^2_v}^{\frac{1}{2}}, 
%	\end{multline}and 
%	\begin{multline}\label{L3}
%		\|f\|_{L^3_xL^2_v}
%		\le \Big(\int_{\R^3}\|f\|_{L^3_x}^2\,dv\Big)^{\frac{1}{2}}
%		\lesssim \Big(\int_{\R^3}\|\<v\>^{k}f\|_{L^2_x}^{\frac{1}{2}}\|\<v\>^{-k}\na_xf\|_{L^2_x}^{\frac{1}{2}}\,dv\Big)^{\frac{1}{2}}
%		\\
%		\lesssim \|\<v\>^{k}f\|^{\frac{1}{2}}_{L^2_{x}L^2_v}\|\<v\>^{-k}\na_{x}f\|^{\frac{1}{2}}_{L^2_xL^2_v}. 
%	\end{multline}
	By our choice of $w_l(\al,\beta)$ in \eqref{w}, we have 
	\begin{align*}
		w_l(\al,0)\le \<v\>^\gamma w_l(|\al|-1,0)^sw(|\al|-1,1)^{1-s}, \quad w_l(\al,0)=\<v\>^\gamma w_l(|\al|-1,0). 
	\end{align*}
Choosing $\<v\>^{k} = w_l(|\al|-1,0)^{1-s}w_l(|\al|-1,1)^{-(1-s)}$ in \eqref{fff}, we obtain 
\begin{align*}
	\|\<v\>^{-\frac{\gamma}{2}}w_l(\al,0)\partial^{\alpha-\alpha_1}(\II-\PP)f\|_{L^2_{x,v}}
	&\lesssim \|\<v\>^{\frac{\gamma}{2}}w_l(|\al|-1,0)\partial^{\alpha-\alpha_1}(\II-\PP)f\|_{L^2_{x}H^s_v}\\
	&\quad+  \|\<v\>^{\frac{\gamma}{2}}w_l(|\al|-1,1)\partial^{\alpha-\alpha_1}(\II-\PP)f\|_{L^2_{x}H^{1+s}_v}\\
	&\le \sqrt{\D_{K,l}}, 
\end{align*}
when $|\al_1|=1$. When $|\al_1|=2$, we have 
\begin{align*}
	\|\<v\>^{-\frac{\gamma}{2}}w_l(\al,0)\partial^{\alpha-\alpha_1}(\II-\PP)f\|_{L^6_xL^2_v}&\le 
	\|\<v\>^{\frac{\gamma}{2}}w_l(|\al|-1,0)\partial^{\alpha-\alpha_1}\na_x(\II-\PP)f\|_{L^2_{x,v}}\\
	&\le \sqrt{\D_{K,l}}. 
\end{align*}
When $3\le |\al_1|\le K$, we have 
\begin{align*}
	\|\<v\>^{-\frac{\gamma}{2}}w_l(\al,0)\partial^{\alpha-\alpha_1}(\II-\PP)f\|_{L^\infty_xL^2_v}&\le 
	\|\<v\>^{\frac{\gamma}{2}}w_l(|\al|-1,0)\partial^{\alpha-\alpha_1}\na_x(\II-\PP)f\|_{H^1_xL^2_v}\\
	&\le \sqrt{\D_{K,l}}. 
\end{align*}
Combining the above estimates, we have 
\begin{align*}
	J &\lesssim \Big(\sum_{|\al_1|=1}\|\pa^{\al_1}\na_x\phi\|_{L^\infty_x}\|\<v\>^{-\frac{\gamma}{2}}w_l(\al,0)\partial^{\alpha-\alpha_1}(\II-\PP)f\|_{L^2_{x,v}}\\
	&\qquad+\sum_{|\al_1|=2}\|\pa^{\al_1}\na_x\phi\|_{L^3_x}\|\<v\>^{-\frac{\gamma}{2}}w_l(\al,0)\partial^{\alpha-\alpha_1}(\II-\PP)f\|_{L^6_{x}L^2_x}\\
	&\qquad+\sum_{3\le|\al_1|\le K}\|\pa^{\al_1}\na_x\phi\|_{L^2_x}\|\<v\>^{-\frac{\gamma}{2}}w_l(\al,0)\partial^{\alpha-\alpha_1}(\II-\PP)f\|_{L^\infty_{x}L^2_x}\Big)\\&\qquad\qquad\times\|\psi_{2|\alpha|-8}\<v\>^{\frac{\gamma}{2}}w_l(\al,0)\partial^\alpha f_\pm\|_{L^2_{x,v}}\\
	&\lesssim \E_{K,l}^{1/2}\D_{K,l}. 
\end{align*}
Collecting all the above estimates for $I$ and $J$, we obtain \eqref{30a}. 
The proof of \eqref{30b} is the same as \eqref{30a}, and the details are omitted for brevity.

\qe\end{proof}

Next we give some illustration for the Macroscopic estimate; see also \cite{Deng2021b}. 
Recall the projection $\PP$ in \eqref{10}. By multiplying the equation \eqref{7} with $\mu^{1/2}, v_j\mu^{1/2}(j=1,2,3)$ and $\frac{1}{6}(|v|^2-3)\mu^{1/2}$ and then integrating them over $\R^3_v$, we have 
\begin{equation}\label{17}\left\{\begin{aligned}
	&\partial_ta_\pm + \nabla\cdot b + \nabla_x\cdot(v\mu^{1/2},(\II-\PP)f)_{L^2_v} =0,\\
	&\partial_t\big(b_j+(v_j\mu^{1/2},(\II-\PP)f)_{L^2_v}\big)+\partial_j(a_\pm+2c)\mp E_j\\&\qquad+(v_j\mu^{1/2},v\cdot\nabla_x(\II-\PP)f)_{L^2_v} = (L_\pm f+g_\pm,v_j\mu^{1/2})_{L^2_v},\\
	&\partial_t\Big(c+\frac{1}{6}((|v|^2-3)\mu^{1/2},(\II-\PP)f)_{L^2_v}\Big)+\frac{1}{3}\nabla_x\cdot b\\&\qquad + \frac{1}{6}((|v|^2-3)\mu^{1/2},v\cdot\nabla(\II-\PP)f)_{L^2_v} = \frac{1}{6}(L_\pm f+g_\pm,(|v|^2-3)\mu^{1/2})_{L^2_v},
\end{aligned}\right.
\end{equation}
where for brevity, we denote $I=(I_+,I_-)$ with $I_\pm f=f_\pm$ and 
\begin{align*}
%	\label{22a}
		g_\pm = \pm\nabla_x\phi\cdot\nabla_vf_\pm\mp\frac{1}{2}\nabla_x\phi\cdot vf_\pm+\Gamma_\pm(f,f).  
\end{align*}
Notice that $(\P_\pm f,v\mu^{1/2})_{L^2_v}$ and $(\P_\pm f,(|v|^2-3)\mu^{1/2})_{L^2_v}$ is not $0$ in general and similar for $\Gamma_\pm$. Also, we have used
\begin{align*}
	(\pm\nabla_x\phi\cdot\nabla_vf_\pm\mp\frac{1}{2}\nabla_x\phi\cdot vf_\pm,\mu^{1/2})_{L^2_v}=0,
\end{align*}which is obtained by integration by parts on $\na_v$. 
In order to obtain the high-order moments, as in \cite{Duan2011}, we define for $1\le j,k\le 3$ that 
\begin{align*}
	\Theta_{jk}(f_\pm) = ((v_jv_k-1)\mu^{1/2},f_\pm)_{L^2_v},\ \ \Lambda_j(f_\pm) =\frac{1}{10}((|v|^2-5)v_j\mu^{1/2},f_\pm)_{L^2_v}. 
\end{align*}
Then multiplying equation \eqref{7} with the high-order moments $(v_jv_k-1)\mu^{1/2}$ and $\frac{1}{10}(|v|^2-5)v_j\mu^{1/2}$ and integrating over $\R^3_v$, we have 
\begin{equation}\label{18}\left\{
	\begin{aligned}
		&\partial_t\big(\Theta_{jj}((\II-\PP)f)+2c\big) + 2\partial_jb_j = \Theta_{jj}(g_\pm+h_\pm),\\
		&\partial_t\Theta_{jk}((\II-\PP)f)+\partial_jb_k+\partial_kb_j + \nabla_x\cdot(v\mu^{1/2},(\II-\PP)f)_{L^2_v}\\&\qquad\qquad\qquad\qquad = \Theta_{jk}(g_\pm+h_\pm)+(\mu^{1/2},g_\pm)_{L^2_v},\ j\neq k,\\
		&\partial_t\Lambda_j((\II-\PP)f)+\partial_jc = \Lambda_j(g_\pm+h_\pm),
	\end{aligned}\right.
\end{equation}
where 
\begin{align*}
	h_\pm = -v\cdot\nabla_x(\II-\PP)f+L_\pm f. 
\end{align*}
By taking the mean value of every two equations with sign $\pm$ in \eqref{17}, we have 
\begin{equation*}\left\{
	\begin{aligned}
		&\partial_t\Big(\frac{a_++a_-}{2}\Big)+\nabla_x\cdot b = 0,\\
		&\partial_tb_j+\partial_j\Big(\Big(\frac{a_++a_-}{2}\Big)+2c\Big)+\frac{1}{2}\sum_{k=1}^3\partial_k\Theta_{jk}((\I-\P)f\cdot[1,1])
		= \frac{1}{2}(g_++g_-,v_j\mu^{1/2})_{L^2_v},\\
		&\partial_tc+\frac{1}{3}\nabla_x\cdot b + \frac{5}{6}\sum^3_{j=1}\partial_j\Lambda_j((\I-\P)f\cdot[1,1]) = \frac{1}{12}(g_++g_-,(|v|^2-3)\mu^{1/2})_{L^2_v},
	\end{aligned}\right.
\end{equation*}for $1\le j\le3$. Similarly, taking the mean value with $\pm$ of the equation in \eqref{18}, we have 
\begin{equation*}
	\left\{\begin{aligned}
		&\partial_t\Big(\frac{1}{2}\Theta_{jk}((\II-\PP)f\cdot[1,1])+2c\delta_{jk}\Big) + \partial_jb_k+\partial_kb_j = \frac{1}{2}\Theta_{jk}(g_++g_-+h_++h_-),\\
		&\frac{1}{2}\partial_t\Lambda_j((\II-\PP)f\cdot[1,1])+\partial_jc = \frac{1}{2}\Lambda_j(g_++g_-+h_++h_-),
	\end{aligned}\right.
\end{equation*}
for $1\le j,k\le 3$. $\delta_{jk}$ is the Kronecker delta. Moreover, for obtaining the dissipation of the electric field $E$, we take the difference with sign $\pm$ in the first two equations in \eqref{17}, we have 
\begin{equation}\label{21}\left\{
	\begin{aligned}
		&\partial_t(a_+-a_-)+\nabla_x\cdot G=0,\\
		&\partial_tG + \nabla_x(a_+-a_-)-2E+\nabla_x\cdot\Theta((\I-\P)f\cdot[1,-1])\\&\qquad=((g+Lf)\cdot[1,-1],v\mu^{1/2})_{L^2_v},
	\end{aligned}\right.
\end{equation}
where 
\begin{align*}
%	\label{27aa}
	G = (v\mu^{1/2},(\I-\P)f\cdot[1,-1])_{L^2_v}.
\end{align*}
Recall that $E=-\nabla_x\phi$. Then by equation \eqref{8}, we have 
\begin{align}\label{16}
	\nabla_x\cdot E = a_+-a_-. 
\end{align}

\section{Regularity}\label{sec5}
In this section, we will prove the smoothing effect of solutions to Vlasov-Poisson-Boltzmann system with lower order initial data. 
Let $K\ge 4$ and $l\ge 0$. The Vlasov-Poisson-Boltzmann system reads  
\begin{equation}\label{16?}
\left\{\begin{aligned}
	&\partial_tf_\pm + v_i\partial^{e_i}f_\pm \pm \frac{1}{2}\partial^{e_i}\phi v_if_\pm  \mp\partial^{e_i}\phi\partial_{e_i}f_\pm \pm \partial^{e_i}\phi  v_i\mu^{1/2} - L_\pm f = \Gamma_{\pm}(f,f),\\
	&-\Delta_x \phi = \int_{\Rd}(f_+-f_-)\mu^{1/2}\,dv, \quad \phi\to 0\text{ as }|x|\to\infty,\\
	&f_\pm|_{t=0} = f_{0,\pm}. 
\end{aligned}\right.
\end{equation}The index appearing in both superscript and subscript means the summation. Our goal is to obtain the $a$ $priori$ estimate from these equations. 
%For this, we define the energy 
%	\begin{equation}
%	X(t) = \sup_{0\le\tau\le t}\E_{K,l+l_1}+\sup_{0\le\tau\le t}(1+\tau)^{3/2}\E_{K,l}(\tau)+\sup_{0\le\tau\le t}(1+t)^{\frac{3}{2}+p}\E^h_{K,l}(\tau),
%\end{equation} and suppose that the Cauchy problem \eqref{16?} admits a smooth solution $f(t,x,v)$ over $0\le t\le T$ for $0<T\le\infty$, and the solution $f(t,x,v)$ satisfies 
%\begin{align}\label{52}
%	\sup_{0\le t\le T}X(t)\le \delta_0,
%\end{align} where $\delta_0$ is a suitably small constant. 
In order to extract the smoothing estimate, we let $N=N(\alpha,\beta)>0$ be a large number chosen later. Assume $T\in(0,1]$, $t\in[0,T]$ and 
\begin{equation}\label{93}
	\psi=t^{N},\quad\psi_k=\left\{\begin{aligned}
		1, \text{  if $k\le 0$},\\
		\psi^k, \text{ if $k>0$}. 
	\end{aligned}\right.
\end{equation}
is this section. Then $|\partial_t\psi_k|\lesssim \psi_{k-1/N}$. Let $f$ be the smooth solution to \eqref{7}-\eqref{9} over $0\le t\le T$ and assume the $a$ $priori$ assumption 
\begin{align}\label{priori1}
	\sup_{0\le t\le T}\E_{K,l}(t)\le \delta_0,
\end{align}where $\delta_0\in(0,1)$ is a suitably small constant.
Under this assumption, we can derive a simple fact that 
\begin{align*}
	\|\phi\|_{L^\infty}\lesssim\|\phi\|_{H^2_x}\le \delta_0, \quad \|e^{\pm\phi}\|_{L^\infty}\approx 1.
\end{align*}
Also, by equation \eqref{21}$_1$ and Gagliardo–Nirenberg interpolation inequality (cf. \cite[Theorem 12.83]{Leoni2017}), we have 
\begin{equation}\label{34a}
	\partial_t\phi = -\Delta_x^{-1}\partial_t(a_+-a_-)=\Delta_x^{-1}\nabla_x\cdot G,
\end{equation}
\begin{equation}\label{34}
	\|\partial_t\phi\|_{L^\infty}\lesssim  \|\nabla_x\partial_t\phi\|^{1/2}_{L^2_x}\|\nabla^2_x\partial_t\phi\|^{1/2}_{L^2_x}\lesssim \|\nabla_x G\|_{H^1_x}\lesssim \|(\I-\P)f\|_{L^2_vH^1_x} \lesssim (\E_{K,l})^{1/2}(t). 
\end{equation}

\begin{Thm}\label{lem51}Assume $-3<\gamma\le-2s$, $0<s<1$, $K\ge 4$, $l\ge 0$.
	Let $f$ be the solution to \eqref{7}-\eqref{9} satisfying that
	\begin{align*}
%		\label{100}
		\epsilon^2_1 = \E_{4,l}(0), \quad
		\sup_{0\le t\le T}\|\<v\>^{C_{K,l}}f(t)\|^2_{L^2_{v,x}}<\infty, 
		\end{align*}for some large constant $C_{K,l}>0$ depending on $K,l$. 
	Then there exists $t_0\in(0,1)$ such that  
	\begin{align*}
		\sup_{0\le t\le t_0}\E_{K,l}(t)\le C_{K,l} \epsilon^2_1.
	\end{align*}
\end{Thm}
	The reason of choosing $\psi_{|\alpha|+|\beta|-4}$ in $\eqref{Defe}$ is that whenever $K\ge 4$, the initial value $\E_{K,l}(0)=\E_{4,l}(0)$, since $\psi_{|\alpha|+|\beta|-4}|_{t=0}=0$ whenever $|\alpha|+|\beta|\ge 5$. In order to prove Theorem \ref{lem51}, we give the following {\it a priori} estimates. 

\begin{Lem}\label{thm41}For any $l\ge 0$, there is $\E_{K,l}$ satisfying \eqref{Defe} such that for $0\le t\le T$,
\begin{align}\label{72}
	\partial_t\E_{K,l}(t)+\lambda\D_{K,l}(t) \lesssim \|\partial_t\phi\|_{L^\infty_x}\E_{K,l}(t)+\E_{K,l} + \sum_{|\alpha|+|\beta|\le K}\|\psi_{|\alpha|+|\beta|-4-\frac{1}{2N}}w_l(\al,\beta)\partial^\alpha_\beta f\|^2_{L^2_{v,x}}.
\end{align}where $D_{K,l}$ is defined by \eqref{Defd}. 
\end{Lem}
\begin{proof}
	For any $K\ge 4$ being the total derivative of $v,x$, we let $|\alpha|+|\beta|\le K$.
	On one hand, we apply $\partial^\alpha$ to equation \eqref{16?} to get 
	\begin{equation}\begin{aligned}\label{35}
		&\quad\,\partial_t\partial^{\alpha} f_\pm + v_i\partial^{e_i+\alpha} f_\pm \pm \frac{1}{2}\sum_{\substack{\alpha_1\le\alpha}}{}\partial^{e_i+\alpha_1}\phi v_i\partial^{\alpha-\alpha_1}f_\pm \\ &\qquad\mp\sum_{\substack{\alpha_1\le\alpha}}{}\partial^{e_i+\alpha_1}\phi\partial^{\alpha-\alpha_1}_{e_i} f_\pm \pm \partial^{e_i+\alpha}\phi v_i\mu^{1/2} - \partial^{\alpha} L_\pm  f =
		\partial^{\alpha} \Gamma_{\pm}(f,f).\end{aligned}
\end{equation}
On the other hand, we apply $\partial^{\alpha}_\beta$ to equation \eqref{16?}. Then,  
	\begin{align}\label{36}
		&\notag\quad\,\partial_t\partial^{\alpha}_\beta f_\pm + \sum_{\beta_1\le \beta}{}\partial_{\beta_1}v_i\partial^{e_i+\alpha}_{\beta-\beta_1}f_\pm \pm \frac{1}{2}\sum_{\substack{\alpha_1\le\alpha}}\sum_{\beta_1\le\beta}{}\partial^{e_i+\alpha_1}\phi \partial_{\beta_1}v_i\partial^{\alpha-\alpha_1}_{\beta-\beta_1}f_\pm \\ &\qquad\mp\sum_{\substack{\alpha_1\le\alpha}}{}\partial^{e_i+\alpha_1}\phi\partial^{\alpha-\alpha_1}_{\beta+e_i}f_\pm \pm \partial^{e_i+\alpha}\phi \partial_\beta(v_i\mu^{1/2}) - \partial^{\alpha}_\beta L_\pm f =
		\partial^{\alpha}_\beta \Gamma_{\pm}(f,f).
	\end{align}	
%\begin{align}\label{36}
%	&\quad\,\partial_t\partial^{\alpha}_\beta (\II-\PP)f + \sum_{\beta_1\le \beta}{}\partial_{\beta_1}v_i\partial^{e_i+\alpha}_{\beta-\beta_1}(\II-\PP)f \notag\\
%	&\notag\qquad\pm \frac{1}{2}\sum_{\substack{\alpha_1\le\alpha}}\sum_{\beta_1\le\beta}\partial^{e_i+\alpha_1}\phi \partial_{\beta_1}v_i\partial^{\alpha-\alpha_1}_{\beta-\beta_1}(\II-\PP)f \\ &\qquad\mp\sum_{\substack{\alpha_1\le\alpha}}\partial^{e_i+\alpha_1}\phi\partial^{\alpha-\alpha_1}_{\beta+e_i}(\II-\PP)f \pm \partial^{e_i+\alpha}\phi \partial_\beta(v_i\mu^{1/2}) - \partial^{\alpha}_\beta L_\pm (\I-\P)f \\
%	&\notag= -\partial_t\partial^{\alpha}_\beta \PP f + \sum_{\beta_1\le \beta}{}\partial_{\beta_1}v_i\partial^{e_i+\alpha}_{\beta-\beta_1}\PP f \mp \frac{1}{2}\sum_{\substack{\alpha_1\le\alpha}}\sum_{\beta_1\le\beta}\partial^{e_i+\alpha_1}\phi \partial_{\beta_1}v_i\partial^{\alpha-\alpha_1}_{\beta-\beta_1}\PP f \\ &\notag\qquad\mp\sum_{\substack{\alpha_1\le\alpha}}\partial^{e_i+\alpha_1}\phi\partial^{\alpha-\alpha_1}_{\beta+e_i}\PP f
%	+
%	\partial^{\alpha}_\beta \Gamma_{\pm}(f,f).
%\end{align}

{\bf Step 1. Estimate without weight.}
For the estimate without weight, we take the case $|\alpha|\le K$ and $\beta=0$. This case is for obtaining the term $\|\partial^\alpha\nabla_x\phi\|^2_{L^2_x}$ on the left hand side of the energy inequality. Taking inner product of equation \eqref{35} with $\psi_{2|\alpha|-8}e^{\pm\phi}\partial^{\alpha} f_\pm$ over $\R^3_v\times\R^3_x$, we have   
\begin{align*}
%	\label{45}
	&\notag\quad\,\Big(\partial_t\partial^{\alpha} f_\pm,\psi_{2|\alpha|-8}e^{\pm\phi}\partial^{\alpha} f_\pm\Big)_{L^2_{v,x}}
	+ \Big(v_i\partial^{e_i+\alpha}f_\pm,\psi_{2|\alpha|-8}e^{\pm\phi}\partial^{\alpha} f_\pm\Big)_{L^2_{v,x}}\\ 
	&\notag\pm \Big(\frac{1}{2}\sum_{\substack{\alpha_1\le\alpha}}{}\partial^{e_i+\alpha_1}\phi v_i\partial^{\alpha-\alpha_1}f_\pm,\psi_{2|\alpha|-8}e^{\pm\phi}\partial^{\alpha} f_\pm\Big)_{L^2_{v,x}} \\ 
	&\mp
	\Big(\sum_{\substack{\alpha_1\le\alpha}}{}\partial^{e_i+\alpha_1}\phi\partial^{\alpha-\alpha_1}_{e_i}f_\pm,\psi_{2|\alpha|-8}e^{\pm\phi}\partial^{\alpha} f_\pm\Big) _{L^2_{v,x}}\\
	&\notag\pm \Big(\partial^{e_i+\alpha}\phi v_i\mu^{1/2},\psi_{2|\alpha|-8}e^{\pm\phi}\partial^{\alpha} f_\pm\Big)_{L^2_{v,x}} 
	- \Big(\partial^{\alpha} L_\pm f,\psi_{2|\alpha|-8}e^{\pm\phi}\partial^{\alpha} f_\pm\Big)_{L^2_{v,x}}\\ 
	&\notag= \Big(\partial^{\alpha} \Gamma_{\pm}(f,f),\psi_{2|\alpha|-8}e^{\pm\phi}\partial^{\alpha} f_\pm\Big)_{L^2_{v,x}}.
\end{align*}
Now we take the summation on $\pm$ and real part, and denote these resulting terms by $I_1$ to $I_7$. In the following we estimate them term by term. 
For the term $I_1$, 
\begin{align}\label{37}
	I_1 &=\notag \frac{1}{2}\partial_t\sum_{\pm}\|e^{\frac{\pm\phi}{2}}\psi_{|\alpha|-4}\partial^{\alpha} f_\pm\|^2_{L^2_{v,x}} \mp \Re\sum_{\pm}\frac{1}{2}(\partial_t\phi e^{\pm\phi}\partial^{\alpha} f_\pm, \psi_{2|\alpha|-8}\partial^{\alpha} f_\pm)_{L^2_{v,x}}\\&\qquad-\Re\sum_{\pm}(\partial_t(\psi_{|\alpha|-4})\partial^{\alpha} f_\pm,\psi_{|\alpha|-4} e^{\pm\phi}\partial^{\alpha} f_\pm)_{L^2_{v,x}}. 
\end{align}
The second term on the right hand side of \eqref{37} is estimated as 
\begin{align*}
%	\label{46}
	\Big|\frac{1}{2}(\partial_t\phi \psi_{2|\alpha|-8}e^{\pm\phi}\partial^{\alpha} f_\pm, \partial^{\alpha} f_\pm)_{L^2_{v,x}}\Big|\lesssim \|\partial_t\phi\|_{L^\infty}\|\psi_{|\alpha|-4}\partial^\alpha f_\pm\|^2_{L^2_{v,x}}\lesssim \|\partial_t\phi\|_{L^\infty}\E_{K,l}(t)
\end{align*}
The third right-hand term of \eqref{37} is estimated as 
	\begin{align}
	|(\partial_t(\psi_{|\alpha|-4})\partial^{\alpha} f_\pm,\psi_{|\alpha|-4} e^{\pm\phi}\partial^{\alpha} f_\pm)_{L^2_{v,x}}|&\lesssim \|\psi_{|\alpha|-4-\frac{1}{2N}}\partial^\alpha f\|^2_{L^2_{v,x}}.\notag
\end{align}

For the second term $I_2$, we will combine it with $I_3$ and $\alpha_1=0$. It turns out that the sum is zero. This is what $e^{\pm\phi}$ designed for, cf. \cite{Guo2012}. Taking integration by parts on $x$, one has  
\begin{align}\label{48aa}
	&\quad\,\Big(v_i\partial^{e_i+\alpha}f_\pm,\psi_{2|\alpha|-8}e^{\pm\phi}\partial^{\alpha} f_\pm\Big)_{L^2_{v,x}}
	\pm \Big(\frac{1}{2}\partial^{e_i}\phi v_i\partial^{\alpha}f_\pm,\psi_{2|\alpha|-8}e^{\pm\phi}\partial^{\alpha} f_\pm\Big)_{L^2_{v,x}}=0.
\end{align}

For the left terms in $I_3$, the weight will be used. In this case, $|\alpha_1|\ge 1$ and by Lemma \ref{Lem26}, it's bounded above by $\E^{1/2}_{K,l}\D_{K,l}$. 
Using Lemma \ref{Lem27}, the term $I_4$ is also bounded above by $\E^{1/2}_{K,l}\D_{K,l}$.

For the term $I_5$, we will decompose $e^{\pm\phi}$ into $(e^{\pm\phi}-1)$ and $1$. Recall equation \eqref{16} and \eqref{21}. For the part of $1$, we have 
\begin{align*}
%	\label{43}\notag
	\sum_{\pm}\pm\Re\big(\partial^{e_i+\alpha}\phi v_i\mu^{1/2},\psi_{2|\alpha|-8}\partial^{\alpha} f_\pm\big)_{L^2_{v,x}} 
    &=\notag -\Re\big(\partial^{\alpha}\phi,\psi_{2|\alpha|-8}\partial^{\alpha} \nabla_x\cdot G\big)_{L^2_{x}}\\
    &=\notag \Re\big(\partial^{\alpha}\phi,\psi_{2|\alpha|-8}\partial^{\alpha} \partial_t(a_+-a_-)\big)_{L^2_{x}}\\
    &= \frac{1}{2}\partial_t\|\psi_{|\alpha|-4}\partial^{\alpha}\nabla_x\phi\|_{L^2_x}^2.
\end{align*}
For the part of $(e^{\pm\phi}-1)$, notice that 
\begin{align*}
	|e^{\pm\phi}-1|\lesssim \|\phi\|_{L^\infty}\lesssim \|\nabla_x\phi\|_{H^1_x}.
\end{align*}Then, 
\begin{align*}
	&\quad\,\Big|\sum_{\pm}\pm\Re\Big(\partial^{e_i+\alpha}\phi v_i\mu^{1/2},(e^{\pm\phi}-1)\psi_{2|\alpha|-8}\partial^{\alpha} f_\pm\Big)_{L^2_{v,x}}\Big|\notag\\
	&\lesssim \|\nabla_x\phi\|_{H^1_x}\sum_{|\alpha|\le K}\|\partial^\alpha\nabla_x\phi\|_{L^2_{v,x}}\sum_{|\alpha|\le K}\|\partial^\alpha(\I-\P)f\|_{L^2_{v,x}}\\
	&\lesssim \E^{1/2}_{K,l}(t)\D_{K,l}(t).\notag
\end{align*}
For the term $I_6$, since $L_\pm$ commutes with $\partial^{\alpha}$ and $e^{\pm\phi}$, by Lemma \ref{lemmaL}, we have 
\begin{align*}
	I_6 = - \sum_{\pm}\Big(\partial^{\alpha} L_\pm f,\psi_{2|\alpha|-8}e^{\pm\phi}\partial^{\alpha} f_\pm\Big)_{L^2_{v,x}}\ge \lambda \sum_{\pm}\|\psi_{|\alpha|-4}\partial^{\alpha}(\II-\PP) f\|_{L^2_xL^2_D}^2. 
\end{align*}
For the term $I_7$, by Lemma \ref{lemmag}, we have 
\begin{align*}
	|I_7|&= \Big|\sum_{\pm}\Big(\partial^{\alpha} \Gamma_{\pm}(f,f),\psi_{2|\alpha|-8}e^{\pm\phi}\partial^{\alpha} f_\pm\Big)_{L^2_{v,x}}\Big|\lesssim\E^{1/2}_{K,l}(t)\D_{K,l}(t).
\end{align*}
Therefore, combining all the estimate above and take the summation on $|\alpha|\le K$, we conclude that, 
\begin{equation}\label{47}
	\begin{aligned}
		&\quad\,\frac{1}{2}\partial_t\sum_{\pm}\sum_{|\alpha|\le K}\Big(\|\psi_{|\alpha|-4}e^{\frac{\pm\phi}{2}}\partial^{\alpha} f_\pm\|_{L^2_{v,x}} +
		\|\psi_{|\alpha|-4}\partial^{\alpha}\nabla_x\phi\|_{L^2_x}^2\Big)\\&\qquad + \lambda \sum_{\pm}\sum_{|\alpha|\le K}\|\psi_{|\alpha|-4}\partial^{\alpha} (\II-\PP)f\|_{L^2_xL^2_D}^2\\
 &\lesssim \|\partial_t\phi\|_{L^\infty}\E_{K,l}(t)+\E^{1/2}_{K,l}(t)\D_{K,l}(t)+\sum_{|\alpha|\le K}\|\psi_{|\alpha|-4-\frac{1}{2N}}w_l(|\al|,0)\partial^\alpha f\|^2_{L^2_{v,x}}.
	\end{aligned}
\end{equation}

{\bf Step 2. Estimate with weight on the mixed derivatives.}
Let $K\ge 4$, $|\alpha|+|\beta|\le K$. Taking inner product of equation \eqref{36} with  $\psi_{2|\alpha|+2|\beta|-8}e^{\pm\phi}w^2_l(\al,\beta)\partial^{\alpha}_\beta f_\pm$ over $\R^3_v\times\R^3_x$, one has
\begin{align*}
	&\quad\,\Big(\partial_t\partial^{\alpha}_\beta  f,e^{\pm\phi}\psi_{2|\alpha|+2|\beta|-8}w^2_l(\al,\beta)\partial^{\alpha}_\beta  f\Big)_{L^2_{v,x}}\\
	&\qquad + \Big(\sum_{\beta_1\le \beta}{}\partial_{\beta_1}v_i\partial^{e_i+\alpha}_{\beta-\beta_1} f,e^{\pm\phi}\psi_{2|\alpha|+2|\beta|-8}w^2_l(\al,\beta)\partial^{\alpha}_\beta  f\Big)_{L^2_{v,x}} \\
	&\qquad\pm \Big(\frac{1}{2}\sum_{\substack{\alpha_1\le\alpha\\\beta_1\le\beta}}{}\partial^{e_i+\alpha_1}\phi \partial_{\beta_1}v_i\partial^{\alpha-\alpha_1}_{\beta-\beta_1} f,e^{\pm\phi}\psi_{2|\alpha|+2|\beta|-8}w^2_l(\al,\beta)\partial^{\alpha}_\beta  f\Big)_{L^2_{v,x}} \\ &\qquad\mp\Big(\sum_{\substack{\alpha_1\le\alpha}}{}\partial^{e_i+\alpha_1}\phi\partial^{\alpha-\alpha_1}_{\beta+e_i} f,e^{\pm\phi}\psi_{2|\alpha|+2|\beta|-8}w^2_l(\al,\beta)\partial^{\alpha}_\beta  f\Big)_{L^2_{v,x}}\\
	&\qquad \pm \Big(\partial^{e_i+\alpha}\phi \partial_\beta(v_i\mu^{1/2}),e^{\pm\phi}\psi_{2|\alpha|+2|\beta|-8}w^2_l(\al,\beta)\partial^{\alpha}_\beta  f\Big)_{L^2_{v,x}}\\
	&\qquad - \Big(\partial^{\alpha}_\beta L_\pm f,e^{\pm\phi}\psi_{2|\alpha|+2|\beta|-8}w^2_l(\al,\beta)\partial^{\alpha}_\beta  f\Big)_{L^2_{v,x}} \\
	&= \Big(\partial^{\alpha}_\beta \Gamma_{\pm}(f,f),e^{\pm\phi}\psi_{2|\alpha|+2|\beta|-8}w^2_l(\al,\beta)\partial^{\alpha}_\beta  f\Big)_{L^2_{v,x}}.
\end{align*}
Now we denote these terms with summation $\sum_{\pm}$ by $J_1$ to $J_{7}$ and estimate them term by term. 
The estimate of $J_1$ to $J_4$ are similar to $I_1$ to $I_4$. 
For $J_1$, we have 
\begin{align*}
	J_1 &\ge \partial_t\sum_{\pm}\|e^{\frac{\pm\phi}{2}}\psi_{|\alpha|+|\beta|-4}w_l(\al,\beta)\partial^{\alpha}_\beta f_\pm\|_{L^2_{v,x}} - C\|\partial_t\phi\|_{L^\infty}\E_{K,l}(t)\\&\qquad -\sum_{\pm}\|\psi_{|\alpha|+|\beta|-4-\frac{1}{2N}}w_l(\al,\beta)\partial^\alpha_\beta f_\pm\|^2_{L^2_{v,x}} ,
\end{align*}
Similar to \eqref{48aa}, $J_2$ and $J_3$ with $\alpha_1=0$ are canceled by using integration by parts. Using Lemma \ref{Lem26} and Lemma \ref{Lem27}, the left case $\alpha_1\neq 0$ in $J_3$ together with $J_4$ are bounded above by $\E^{1/2}_{K,l}(t)\D_{K,l}(t)$.
For the term $J_5$, we only need an upper bound: for any $\eta>0$, 
\begin{align*}\notag
	|J_5| &= 
	\Big|\sum_{\pm}\pm \Big(\partial^{e_i+\alpha}\phi \partial_\beta(v_i\mu^{1/2}),\psi_{2|\alpha|+2|\beta|-8}e^{\pm\phi}w^2_l(\al,\beta)\partial^{\alpha}_\beta f_\pm\Big)_{L^2_{v,x}}\Big|\\
	&\lesssim \eta\sum_\pm\|\psi_{|\alpha|+|\beta|-4}w_l(\al,\beta)\partial^{\alpha}_{\beta}f_\pm\|^2_{L^2_xL^2_D}+C_\eta\|\psi_{|\alpha|-4}\partial^\alpha\nabla_x\phi\|_{L^2_{v,x}}^2.
\end{align*}
Notice that $\|\psi_{|\alpha|-4}\partial^\alpha\nabla_x\phi\|_{L^2_{v,x}}^2$ is bounded above by $\E_{K,l}$. For the term $J_6$, since $L_\pm$ commutes with $e^{\pm\phi}$, by Lemma \ref{lemmaL}, we have 
\begin{align*}
	J_6 &= - \sum_{\pm}\Big(\partial^{\alpha}_\beta L_\pm f,\psi_{2|\alpha|+2|\beta|-8}e^{\pm\phi}w^2_l(\al,\beta)\partial^{\alpha}_\beta f_\pm\Big)_{L^2_{v,x}}\\&\ge \lambda\sum_{\pm} \|\psi_{|\alpha|+|\beta|-4}e^{\frac{\pm\phi}{2}}w_l(\al,\beta)\partial^{\alpha}_\beta f_\pm\|_{L^2_xL^2_D}^2-C_\eta\sum_{\pm}\|\partial^\alpha f_\pm\|^2_{L^2_xL^2_D}\\
	&\qquad -\eta\sum_{\pm}\sum_{|\beta_1|\le|\beta|}\|\psi_{|\alpha|+|\beta|-4}e^{\frac{\pm\phi}{2}}w_l(\al,\beta_1)\partial^{\alpha}_{\beta_1} f_\pm\|^2_{L^2_xL^2_D},
\end{align*}for any $\eta>0$. 
Here we use the fact that $\|w_l(\al,\beta)(\cdot)\|_{L^2(B_{C_\eta})}\lesssim \|\cdot\|_{L^2_D}$. 
The term $J_{7}$, by Lemma \ref{lemmag}, is bounded above by  $\E^{1/2}_{K,l}\D_{K,l}+\E_{K,l}\D^{1/2}_{K,l}\lesssim (\E^{1/2}_{K,l}+\E_{K,l})\D_{K,l}+\E_{K,l}$.

Combining all the above estimate, taking summation on $|\alpha|+|\beta|\le K$ and letting $\eta$ sufficiently small, we have 
\begin{multline}\label{111ab}
	\frac{1}{2}\partial_t\sum_{|\alpha|+|\beta|\le K,\,\pm}\|e^{\frac{\pm\phi}{2}}\psi_{|\alpha|+|\beta|-4}\partial^\alpha_\beta f_\pm\|^2_{L^2_{v,x}}+\lambda \sum_{|\alpha|+|\beta|\le K,\,\pm}\|\psi_{|\alpha|+|\beta|-4}w_l(\al,\beta)\partial^\alpha_\beta f_\pm\|^2_{L^2_xL^2_D} \\\lesssim \|\partial_t\phi\|_{L^\infty_x}\E_{K,l}(t) + \sum_{|\alpha|+|\beta|\le K}\|\psi_{|\alpha|+|\beta|-4-\frac{1}{2N}}w_l(\al,\beta)\partial^\alpha_\beta f\|^2_{L^2_{v,x}}+(\E^{1/2}_{K,l}+\E_{K,l})\D_{K,l}+\E_{K,l}.
\end{multline}
Together with \eqref{priori1}, taking combination $\eqref{47}+\eqref{111ab}$, we have 
\begin{multline}\label{65}
	\partial_t\E_{K,l}(t)+\lambda\D_{K,l}(t) \lesssim \|\partial_t\phi\|_{L^\infty_x}\E_{K,l}(t)\\+\E_{K,l} + \sum_{|\alpha|+|\beta|\le K}\|\psi_{|\alpha|+|\beta|-4-\frac{1}{2N}}w_l(\al,\beta)\partial^\alpha_\beta f\|^2_{L^2_{v,x}},
\end{multline}
where we let 
\begin{align}\label{EE}
\E_{K,l}(t) = \sum_\pm\sum_{|\alpha|+|\beta|\le K}\|e^{\frac{\pm\phi}{2}}\psi_{|\alpha|+|\beta|-4}\partial^\alpha_\beta f_\pm\|^2_{L^2_{v,x}} + \sum_{|\alpha|\le K}
\|\psi_{|\alpha|-4}\partial^{\alpha}\nabla_x\phi\|_{L^2_x}^2.
\end{align}
It's straightforward to show that $\E_{K,l}$ satisfies \eqref{Defe}.  
Notice that there's $\|\psi_{|\alpha|-4}\partial^\alpha E(t)\|^2_{L^2_x}$  in $\E_{K,l}$ on the right hand side of \eqref{65}, and hence we can put $\|\psi_{|\alpha|-4}\partial^\alpha E(t)\|^2_{L^2_x}$, which is in $\D_{K,l}$, on the left hand side. 
\qe\end{proof}

 Therefore, now it suffices to control the last term in \eqref{72}. 
\begin{Lem}\label{Lem33}	Let $K\ge 4$ and $f$ to be the solution to \eqref{7}-\eqref{9} and assume the same assumption as in Lemma \ref{thm41}. It holds that for any $0<\delta<1$ and multi-indices $|\alpha|+|\beta|\le K$, 
		\begin{align}\label{112}
		&\notag\quad\,\|\psi_{|\alpha|+|\beta|-4-\frac{1}{2N}}w_l(\al,\beta)\partial^\alpha_\beta f\|^2_{L^2_{v,x}}\\
		&\lesssim \delta^2\partial_t\big(-\psi_{2|\alpha|+2|\beta|-8}w_l(\al,\beta)\partial^\alpha_\beta f_\pm e^{\frac{\pm\phi}{2}},(\theta^ww_l(\al,\beta){(\partial^\alpha_\beta f_\pm e^{\frac{\pm\phi}{2}})^\wedge})^\vee\big)_{L^2_{v,x}}\\
		&\notag\qquad+\delta^2\big(\D_{K,l}+(\E^{1/2}_{K,l}+\E_{K,l})\D_{K,l}+\|\partial_t\phi\|_{L^\infty_x}\E_{K,l}(t)+\E_{K,l}\big)+C_\delta\|\<v\>^{C_{K,l}}f\|_{L^2_{v,x}}^2,
	\end{align}
	where $\theta^w=\theta^w(v,D_v)$ and $\theta\in S(1)$ is defined by \eqref{107}. 
\end{Lem}

\begin{proof}
{\bf Step 1.}
To deal with the last term of \eqref{72}, we choose constants 
\begin{align}\label{106b}\notag
	\delta_1=\delta_1(\alpha,\beta)&\in\Big(0,\min\big\{\frac{2s}{1+2s},\frac{1}{2}\big\}\Big],\\ 
	\delta_2= 1-\delta_1&\in\Big[\max\big\{\frac{1}{1+2s},\frac{1}{2}\big\},1\Big),\\ l_0=\gamma\delta_2&<0\notag
\end{align} to be determined later. Let $\chi_0$ to be a smooth cutoff function such that $\chi_0(z)$ equal to $1$ when $|z|<\frac{1}{2}$ and equal to $0$ when $|z|\ge 1$. Define 
\begin{align}\label{98}
	\tilde{b}(v,y) &= \<v\>^{l_0}|y|^{\delta_1},\\
	\chi(v,\eta) &= \chi_0\bigg(\frac{\<\eta\>\<v\>^{l_0}}{|y|^{\delta_2}}\bigg),\notag
\end{align}and
\begin{align}\label{107}
	\theta(v,\eta) = \<v\>^{l_0}|y|^{-1-\delta_2}y\cdot\eta\,\chi(v,\eta).
\end{align}
Notice that for any multi-indices $\alpha,\beta$, 
\begin{align*}
	\psi_{|\alpha|+|\beta|-4-\frac{1}{2N}} = \psi_{|\alpha|-4-\frac{1}{2N}}\psi_{|\beta|-4-\frac{1}{2N}}.
\end{align*}
 If $|\alpha|>4$, we choose $N=N(\alpha)$ such that 
\begin{align}\label{106a}
	-\frac{2N(|\alpha|-4)-1}{2} = -\frac{|\alpha|}{\delta_1}.
\end{align}
Then by the definition \eqref{98} of $\tilde{b}$ and Young's inequality,
\begin{align}\notag
	\psi_{|\alpha|-4-\frac{1}{2N}}&\lesssim \delta \big((\tilde{b}^{1/2})^{\frac{|\alpha|-4-\frac{1}{2N}}{|\alpha|-4}}\psi_{|\alpha|-4-\frac{1}{2N}}\big)^{\frac{|\alpha|-4}{|\alpha|-4-\frac{1}{2N}}}+C_{0,\delta}\big((\tilde{b}^{-1/2})^{\frac{|\alpha|-4-\frac{1}{2N}}{|\alpha|-4}}\big)^{2N(|\alpha|-4)}\\
	&\lesssim \delta\, \tilde{b}^{1/2}\psi_{|\alpha|-4} + C_{0,\delta}(\<v\>^{-\frac{l_0|\alpha|}{\delta_1}}|y|^{-|\alpha|}),\label{100a}
\end{align}where $C_{0,\delta}$ is a large constant depending on $\delta>0$ and $|\alpha|$. 
If $|\alpha|\le 4$, we choose $\eta\in[0,1)$ such that $-\frac{\eta}{2(1-\eta)}=\frac{-|\alpha|}{\delta_1}$. Then 
\begin{align*}
	\psi_{|\alpha|-4-\frac{1}{2N}}=1&\lesssim (\delta^\eta\,\tilde{b}^{\eta/2})^{\frac{1}{\eta}} + (\delta^{-\eta}\tilde{b}^{-\eta/2})^{\frac{1}{1-\eta}}\\
	&\lesssim \delta\,\tilde{b}^{1/2}+C_{0,\delta}\<v\>^{\frac{-l_0|\alpha|}{\delta_1}}|y|^{-|\alpha|}.
\end{align*} 
Thus, taking the Fourier transform $(\cdot)^\wedge$ with respect to $x$, we have 
\begin{align}\label{101}
	&\quad\,\|\psi_{|\alpha|+|\beta|-4-\frac{1}{2N}}w_l(\al,\beta)\partial^\alpha_\beta f\|_{L^2_{v,x}}\notag= \|\psi_{|\alpha|+|\beta|-4-\frac{1}{2N}}w_l(\al,\beta)(\partial^\alpha_\beta f)^\wedge(v,y)\|_{L^2_{v,y}}\\
	&\lesssim \delta\|\psi_{|\alpha|+|\beta|-4}\tilde{b}^{1/2}w_l(\al,\beta)(\partial^\alpha_\beta f)^\wedge(v,y)\|_{L^2_{v,y}} + C_{0,\delta} \|\psi_{|\beta|-4-\frac{1}{2N}}w_l(\al,\beta)\<v\>^{\frac{-l_0|\alpha|}{\delta_1}}\partial_\beta f\|_{L^2_{v,x}}.
\end{align}
To deal with the second right-hand term of \eqref{101}, we use a similar interpolation on $\tilde{a}^{1/2}$. In fact, if $|\beta|>4$, we have 
\begin{align}\label{107b}\notag
	\psi_{|\beta|-4-\frac{1}{2N}}\<v\>^{\frac{-l_0|\alpha|}{\delta_1}}&\lesssim \Big(\psi_{|\beta|-4-\frac{1}{2N}}\big(\frac{\delta}{C_{0,\delta}}\tilde{a}^{1/2}\big)^{\frac{|\beta|-4-\frac{1}{2N}}{|\beta|-4}}\Big)^{\frac{|\beta|-4}{|\beta|-4-\frac{1}{2N}}}\\
	&\qquad\notag+((C_{0,\delta}\delta^{-1}\tilde{a}^{-1/2})^{\frac{|\beta|-4-\frac{1}{2N}}{|\beta|-4}}\<v\>^{\frac{-l_0|\alpha|}{\delta_1}})^{2N(|\beta|-4)}\\
	&\lesssim \frac{\delta}{C_{0,\delta}}\psi_{|\beta|-4}\tilde{a}^{1/2}+C_{\delta}\tilde{a}^{-\frac{1}{2}(2N(|\beta|-4)-1)}\<v\>^{C_K},
\end{align}where $C_{0,\delta}$ comes from \eqref{100a} and $C_K$ depends only on $K$.
When $|\alpha|>4$, recalling the definition \eqref{11a} of $\tilde{a}$, \eqref{106a} gives that 
\begin{align*}
	\tilde{a}^{-\frac{1}{2}(2N(|\beta|-4)-1)}\lesssim (\<v\>^{-\gamma}\<\eta\>^{-2s})^{\frac{|\beta|-4}{|\alpha|-4}\big(\frac{2|\alpha|}{\delta_1}+1\big)-1}.
\end{align*}
Now we choose $\delta_1=\delta_1(\alpha,\beta)>0$ sufficiently small such that 
\begin{align*}
%	\label{110}
	-2s\Big(\frac{|\beta|-4}{|\alpha|-4}\Big(\frac{2|\alpha|}{\delta_1}+1\Big)-1\Big)\le -|\beta|.
\end{align*}
Then, 
\begin{align*}
\tilde{a}^{-\frac{1}{2}(2N(|\beta|-4)-1)}\lesssim \<v\>^{C_{K}}\<\eta\>^{-|\beta|}.
\end{align*}
When $|\alpha|\le 4$, $N$ can be arbitrary large. Then we choose $N$ sufficiently large that 
\begin{align*}
	\tilde{a}^{-\frac{1}{2}(2N(|\beta|-4)-1)}\lesssim \<v\>^{C_{K}}\<\eta\>^{-|\beta|}. 
\end{align*}
Thus, \eqref{107b} becomes
\begin{align*}
	\psi_{|\beta|-4-\frac{1}{2N}}\<v\>^{\frac{-l_0|\alpha|}{\delta_1}}\lesssim \frac{\delta}{C_{0,\delta}}\psi_{|\beta|-4}\tilde{a}^{1/2}+C_{\delta}\<v\>^{C_{K}}\<\eta\>^{-|\beta|}.
\end{align*}
If $|\beta|\le 4$, we choose $\eta\in(0,1)$ such that $\frac{-\eta}{2(1-\eta)}=\frac{-|\beta|}{2s}$. Then 
\begin{align*}
	\psi_{|\beta|-4-\frac{1}{2N}}\<v\>^{\frac{-l_0|\alpha|}{\delta_1}}=\<v\>^{\frac{-l_0|\alpha|}{\delta_1}}&\lesssim \frac{\delta}{C_{0,\delta}}\,\tilde{a}^{1/2}+C_\delta(\tilde{a}^{-1/2}\<v\>^{\frac{-l_0|\alpha|}{\delta_1}})^{\frac{\eta}{(1-\eta)}}\\
	&\lesssim \frac{\delta}{C_{0,\delta}}\,\tilde{a}^{1/2}+C_\delta\<v\>^{C_K}\<\eta\>^{-|\beta|}.
\end{align*}
Thus, whenever $|\beta|\le 4$ or $|\beta|>4$, we have $$\psi_{|\beta|-4-\frac{1}{2N}}\<v\>^{\frac{-l_0|\alpha|}{\delta_1}}\in S(\frac{\delta}{C_{0,\delta}}\,\tilde{a}^{1/2}+C_\delta\<v\>^{C_K}\<\eta\>^{-|\beta|})$$ uniformly in $\delta$, as a symbol in $(v,\eta)$. 
Then using Lemma \ref{bound_varepsilon} with respect to $v$, we have 
\begin{align*}
	&\|\psi_{|\beta|-4-\frac{1}{2N}}w_l(\al,\beta)\<v\>^{\frac{-l_0|\alpha|}{\delta_1}}\partial_\beta f\|_{L^2_{v,x}}\\&\quad\lesssim \frac{\delta}{C_{0,\delta}}\|\psi_{|\beta|-4}(\tilde{a}^{1/2})^ww_l(0,\beta_1)\partial_\beta f\|_{L^2_{v,x}}+C_\delta\|\<v\>^{C_{K,l}}f\|_{L^2_{v,x}}\\
	&\quad\lesssim \frac{\delta}{C_{0,\delta}}\D^{1/2}_{K,l}+C_\delta\|\<v\>^{C_{K,l}}f\|_{L^2_{v,x}}. 
\end{align*}
Plugging this into \eqref{101}, we have 
\begin{align}\label{101a}\notag
	&\quad\,\|\psi_{|\alpha|+|\beta|-4-\frac{1}{2N}}w_l(\al,\beta)\partial^\alpha_\beta f\|^2_{L^2_{v,x}}\\&\lesssim \delta^2\|\psi_{|\alpha|+|\beta|-4}\tilde{b}^{1/2}w_l(\al,\beta)(\partial^\alpha_\beta f)^\wedge(v,y)\|^2_{L^2_{v,y}}+\delta^2\D_{K,l}+C_\delta\|\<v\>^{C_{K,l}}f\|_{L^2_{v,x}}^2.
\end{align}
Now it suffices to eliminate the first right-hand term of \eqref{101a}.

\medskip

{\bf Step 2.} 
Recalling \eqref{107}, we regard $\theta$ as a symbol in $(v,\eta)$ with parameter $y$. Then, 
\begin{align*}
	|\theta(v,\eta)| = \<v\>^{l_0}|y|^{-1-\delta_2}|y\cdot\eta|\,\chi(v,\eta)
	&\lesssim 1.
\end{align*}
Direct calculation gives that $\partial^\alpha_v\partial^\beta_\eta\theta\lesssim 1$ and hence $\theta\in S(1)$ as a symbol on $(v,\eta)$. 
On the other hand, regarding the Poisson bracket on $(v,\eta)$ we have 
\begin{align*}
	\{\theta,v\cdot y\}&= \<v\>^{l_0}|y|^{1-\delta_2}+\<v\>^{l_0}|y|^{1-\delta_2}(\chi(v,\eta)-1)+\<v\>^{l_0}|y|^{-1-\delta_2}y\cdot\eta\,\partial_\eta\chi\cdot y\\
	&=: \tilde{b} + R_1 + R_2. 
\end{align*}
Now we claim that $R_1,R_2\in S(\tilde{a})$. Indeed, noticing the support of $\chi-1$, by \eqref{106b} we have 
\begin{align*}
	|R_1|\le \<v\>^{l_0}\<\eta\>^{\frac{1-\delta_2}{\delta_2}}\<v\>^{l_0\frac{1-\delta_2}{\delta_2}}
	&\le \<v\>^{\gamma}\<\eta\>^{2s}\le \tilde{a}.  
\end{align*} 
For $R_2$, since $1-2\delta_2\le 0$, we have 
\begin{align*}
	|R_2|\le\<v\>^{2l_0}|y|^{1-2\delta_2}|\eta|\1_{\<\eta\>\<v\>^{l_0}\le|y|^{\delta_2}}\le \<v\>^{\frac{l_0}{\delta_2}}\<\eta\>^{\frac{1-\delta_2}{\delta_2}}\le \tilde{a}.
\end{align*}Higher derivative estimate can be calculated by Leibniz's formula and hence, $R_1,R_2\in S(\tilde{a})$. Thus, by Lemma \ref{innerproduct} and \eqref{compostion}, we have 
\begin{align}\label{102}\notag
	\|\tilde{b}^{1/2}\widehat{g}(v,y)\|^2_{L^2_{v,y}} &=\big(\tilde{b}(v,y)\widehat{g},\widehat{g}\big)_{L^2_{v,y}} \\\notag
	&= \Re\big(\{\theta,v\cdot y\}^w(v,D_v)\widehat{g},\widehat{g}\big)_{L^2_{v,y}} + \Re ((R_1+R_2)^w(v,D_v)\widehat{g},\widehat{g})_{L^2_{v,y}}\\\notag
	&\le 2\pi\Re\big(iv\cdot y\widehat{g},\theta^w(v,D_v)\widehat{g}\big)_{L^2_{v,y}}+C\|(\tilde{a}^{1/2})^wg\|^2_{L^2_{v,x}}\\
	&\le 2\pi\Re\big(v\cdot \nabla_x{g},(\theta^w\widehat{g})^\vee\big)_{L^2_{v,x}}+C\|(\tilde{a}^{1/2})^wg\|^2_{L^2_{v,x}},
\end{align}
for any $g$ in a suitable smooth space. Here and after, we write  $\theta^w=\theta^w(v,D_v)$.
Note that 
\begin{align*}
	\Re 2\pi\big(iv\cdot y\widehat{g},\theta^w(v,D_v)\widehat{g}\big)_{L^2_{v,y}}
	&= 2\pi\big(iv\cdot y\widehat{g},\theta^w(v,D_v)\widehat{g}\big)_{L^2_{v,y}}+2\pi\big(\theta^w(v,D_v)\widehat{g},iv\cdot y\widehat{g}\big)_{L^2_{v,y}}\\
	&= 2\pi\big(i[\theta(v,D_v),v\cdot y]^w\widehat{g},\widehat{g}\big)_{L^2_{v,y}}\\
	&= \big(\{\theta,v\cdot y\}^w(v,D_v)\widehat{g},\widehat{g}\big)_{L^2_{v,y}}, 
\end{align*}
and the Weyl quantization $(\cdot)^w$ is acting on $(v,\eta)$ with parameter $y$. 

Now we let $g = \psi_{|\alpha|+|\beta|-4}w_l(\al,\beta)\partial^\alpha_\beta f_\pm e^{\frac{\pm\phi}{2}}$ in \eqref{102}, then 
\begin{align}\label{104}
	&\notag\quad\,\|\tilde{b}^{1/2}\psi_{|\alpha|+|\beta|-4}w_l(\al,\beta)(\partial^\alpha_\beta f_\pm)^\wedge(v,y) e^{\frac{\pm\phi}{2}}\|_{L^2_{v,x}}\\
	&\lesssim\Re\big(v\cdot \nabla_x{\psi_{|\alpha|+|\beta|-4}w_l(\al,\beta)\partial^\alpha_\beta f}e^{\frac{\pm\phi}{2}},(\theta^w\psi_{|\alpha|+|\beta|-4}w_l(\al,\beta){(\partial^\alpha_\beta f_\pm e^{\frac{\pm\phi}{2}})^\wedge})^\vee\big)_{L^2_{v,x}}+\D_{K,l}\\
	&=: K_0 + \D_{K,l}.\notag
\end{align}
By equation \eqref{7}, we have 
\begin{align*}
	&\quad\,v\cdot\nabla_x(\partial^\alpha_\beta f_\pm e^{\frac{\pm\phi}{2}})\\
	&= v_i\partial^{\alpha+e_i}_\beta f_\pm e^{\frac{\pm\phi}{2}} \pm \frac{1}{2}v_i\partial^{e_i}\phi e^{\frac{\pm\phi}{2}}\partial^\alpha_\beta f_\pm\\
	&= \partial_\beta\big(v_i\partial^{\alpha+e_i}f_\pm e^{\frac{\pm\phi}{2}}\big)-\sum_{0\neq \beta_1\le \beta}{}\partial_{\beta_1}v_i\partial^{\alpha+e_i}_{\beta-\beta_1}f_\pm e^{\frac{\pm\phi}{2}}\pm \frac{1}{2}v_i\partial^{e_i}\phi e^{\frac{\pm\phi}{2}}\partial^\alpha_\beta f_\pm\\
	&=-\partial_t\partial^\alpha_\beta f_\pm e^{\frac{\pm\phi}{2}} \mp \frac{1}{2}\sum_{\alpha_1\le\alpha}\sum_{\beta_1\le\beta}{}\partial^{e_i+\alpha_1}\phi\partial_{\beta_1}v_i\partial^{\alpha-\alpha_1}_{\beta-\beta_1}f_\pm e^{\frac{\pm\phi}{2}}\\
	&\qquad\pm \sum_{\alpha_1\le\alpha}C^{\alpha_1}_\alpha\partial^{e_i+\alpha_1}\phi\partial^{\alpha-\alpha_1}_{\beta+e_i}f_\pm e^{\frac{\pm\phi}{2}} 
	 \mp \partial^{e_i+\alpha}\phi\partial_\beta(v_i\mu^{1/2})e^{\frac{\pm\phi}{2}}+\partial^\alpha_\beta L_\pm fe^{\frac{\pm\phi}{2}}\\
	 &\qquad+\partial^\alpha_\beta\Gamma_\pm(f,f)e^{\frac{\pm\phi}{2}} -\sum_{0\neq \beta_1\le \beta}{}\partial_{\beta_1}v_i\partial^{\alpha+e_i}_{\beta-\beta_1}f_\pm e^{\frac{\pm\phi}{2}}\pm \frac{1}{2}v_i\partial^{e_i}\phi e^{\frac{\pm\phi}{2}}\partial^\alpha_\beta f_\pm
\end{align*}
Thus, 
\begin{align*}
 &K_0 =\psi_{2|\alpha|+2|\beta|-8}\bigg(
 \Re\big(-w_l(\al,\beta)\partial_t\partial^\alpha_\beta f_\pm e^{\frac{\pm\phi}{2}},(\theta^ww_l(\al,\beta){(\partial^\alpha_\beta f_\pm e^{\frac{\pm\phi}{2}})^\wedge})^\vee\big)_{L^2_{v,x}} \\
&\mp \Re\big(w_l(\al,\beta)\frac{1}{2}\sum_{\substack{\alpha_1\le\alpha\\\beta_1\le\beta}}{}\partial^{e_i+\alpha_1}\phi\partial_{\beta_1}v_i\partial^{\alpha-\alpha_1}_{\beta-\beta_1}f_\pm e^{\frac{\pm\phi}{2}},(\theta^ww_l(\al,\beta){(\partial^\alpha_\beta f_\pm e^{\frac{\pm\phi}{2}})^\wedge})^\vee\big)_{L^2_{v,x}} \\
& \pm \Re\big(w_l(\al,\beta)\sum_{\alpha_1\le\alpha}C^{\alpha_1}_\alpha\partial^{e_i+\alpha_1}\phi\partial^{\alpha-\alpha_1}_{\beta+e_i}f_\pm e^{\frac{\pm\phi}{2}},(\theta^ww_l(\al,\beta){(\partial^\alpha_\beta f_\pm e^{\frac{\pm\phi}{2}})^\wedge})^\vee\big)_{L^2_{v,x}} \\
 &\mp \Re\big(w_l(\al,\beta)\partial^{e_i+\alpha}\phi\partial_\beta(v_i\mu^{1/2})e^{\frac{\pm\phi}{2}},(\theta^ww_l(\al,\beta){(\partial^\alpha_\beta f_\pm e^{\frac{\pm\phi}{2}})^\wedge})^\vee\big)_{L^2_{v,x}}
 \\&+\Re\big(w_l(\al,\beta)\partial^\alpha_\beta L_\pm f e^{\frac{\pm\phi}{2}}, (\theta^ww_l(\al,\beta){(\partial^\alpha_\beta f_\pm e^{\frac{\pm\phi}{2}})^\wedge})^\vee\big)_{L^2_{v,x}}  \\ &+\Re\big(w_l(\al,\beta)\partial^\alpha_\beta\Gamma_\pm(f,f)e^{\frac{\pm\phi}{2}},(\theta^ww_l(\al,\beta){(\partial^\alpha_\beta f_\pm e^{\frac{\pm\phi}{2}})^\wedge})^\vee\big)_{L^2_{v,x}}  \\
 &-\Re\big(w_l(\al,\beta)\sum_{0\neq \beta_1\le \beta}{}\partial_{\beta_1}v_i\partial^{\alpha+e_i}_{\beta-\beta_1}f_\pm e^{\frac{\pm\phi}{2}},(\theta^ww_l(\al,\beta){(\partial^\alpha_\beta f_\pm e^{\frac{\pm\phi}{2}})^\wedge})^\vee\big)_{L^2_{v,x}}\\
 &\pm \Re\big(w_l(\al,\beta)\frac{1}{2}v_i\partial^{e_i}\phi e^{\frac{\pm\phi}{2}}\partial^\alpha_\beta f_\pm,(\theta^ww_l(\al,\beta){(\partial^\alpha_\beta f_\pm e^{\frac{\pm\phi}{2}})^\wedge})^\vee\big)_{L^2_{v,x}}\bigg).
\end{align*}
Denote these terms by $K_1$ to $K_8$. Noticing that there's coefficient $\delta$ in \eqref{101a}, we only need to obtain an upper bound for these terms. 
For $K_1$, noticing that $\theta^w$ is self-adjoint, we have 
\begin{align*}
	K_1&\le \frac{1}{2}\partial_t\big(-\psi_{2|\alpha|+2|\beta|-8}w_l(\al,\beta)\partial^\alpha_\beta f_\pm e^{\frac{\pm\phi}{2}},(\theta^ww_l(\al,\beta){(\partial^\alpha_\beta f_\pm e^{\frac{\pm\phi}{2}})^\wedge})^\vee\big)_{L^2_{v,x}}\\
	&\qquad+C\big|\big(-\psi_{2|\alpha|+2|\beta|-8-\frac{1}{N}}w_l(\al,\beta)\partial^\alpha_\beta f_\pm e^{\frac{\pm\phi}{2}},(\theta^ww_l(\al,\beta){(\partial^\alpha_\beta f_\pm e^{\frac{\pm\phi}{2}})^\wedge})^\vee\big)_{L^2_{v,x}}\big|\\
	&\qquad+C\big|\big(\partial_t\phi\psi_{2|\alpha|+2|\beta|-8}w_l(\al,\beta)\partial^\alpha_\beta f_\pm e^{\frac{\pm\phi}{2}},(\theta^ww_l(\al,\beta){(\partial^\alpha_\beta f_\pm e^{\frac{\pm\phi}{2}})^\wedge})^\vee\big)_{L^2_{v,x}}\big|.
\end{align*}
We denote the second and third term on the right hand side by $K_{1,1}$ and $K_{1,2}$. Since $\theta\in S(1)$, $\theta^w$ is a bounded operator on $L^2_{v,y}$. The boundedness of $\theta^w$ will be frequently used in the following without further mentioned. Using the trick from \eqref{100a}-\eqref{101a} to the term for the first $f_\pm$ in $K_{1,1}$, we have 
\begin{align*}
	K_{1,1}\lesssim \delta^2\|\psi_{|\alpha|+|\beta|-4}\tilde{b}^{1/2}w_l(\al,\beta)(\partial^\alpha_\beta f)^\wedge\|_{L^2_{v,y}}^2+\delta^2\D_{K,l}+C_\delta\|\<v\>^{C_{K,l}}f\|_{L^2_{v,x}}^2+\E_{K,l}. 
\end{align*}
The term $K_{1,2}$ is similar to the case $I_1$, i.e.
\begin{align*}
	K_{1,2} \lesssim \|\partial_t\phi\|_{L^\infty_x}\|\psi_{|\alpha|+|\beta|-4}w_l(\al,\beta)\partial^\alpha_\beta f\|_{L^2_{v,x}}^2\lesssim \|\partial_t\phi\|_{L^\infty_x}\E_{K,l}(t). 
\end{align*}
For the term $K_2$ with $\alpha_1=\beta_1=0$, a nice observation is that it's the same as $K_8$ except the sign and hence, they are eliminated. For $K_2$ with $\alpha_1+\beta_1\neq 0$, the order of derivatives for the first $f_\pm$ is less or equal to $K-1$ and hence, the weight can be controlled as $w_l(\al,\beta)\partial_{\beta_1}v_i\lesssim \<v\>^{\gamma}w_l(\al-\al_1,\beta-\beta_1)$. Then similar to Lemma \ref{Lem26}, by noticing $\theta^w$ is bounded on $L^2_{v,y}$, we have 
\begin{align*}
	|K_2+K_8|\lesssim \E^{1/2}_{K,l}\D_{K,l}.
\end{align*}
For $K_3$, when $\alpha_1=0$, noticing $\theta^w$ is self-adjoint, we use integration by parts over $v$ to obtain 
\begin{align*}
|K_3| &= \big|\big(\psi_{2|\alpha|+2|\beta|-8}w_l(\al,\beta)\partial^{e_i}\phi\partial^{\alpha}_{\beta+e_i}f_\pm e^{\frac{\pm\phi}{2}},(\theta^ww_l(\al,\beta){(\partial^\alpha_\beta f_\pm e^{\frac{\pm\phi}{2}})^\wedge})^\vee\big)_{L^2_{v,x}}\big|\\
&\lesssim \big|\big(\psi_{2|\alpha|+2|\beta|-8}\partial_{e_i}(w_l(\al,\beta))\partial^{e_i}\phi\partial^{\alpha}_{\beta}f_\pm e^{\frac{\pm\phi}{2}},(\theta^ww_l(\al,\beta){(\partial^\alpha_\beta f_\pm e^{\frac{\pm\phi}{2}})^\wedge})^\vee\big)_{L^2_{v,x}}\big|\\
&\quad+\big|\big(\psi_{2|\alpha|+2|\beta|-8}w_l(\al,\beta)\partial^{e_i}\phi\partial^{\alpha}_{\beta}f_\pm e^{\frac{\pm\phi}{2}},(\underbrace{[\partial_{e_i},\theta^w]}_{\in S(1)}w_l(\al,\beta){(\partial^\alpha_\beta f_\pm e^{\frac{\pm\phi}{2}})^\wedge})^\vee\big)_{L^2_{v,x}}\big|\\
&\quad+\big|\big(\psi_{2|\alpha|+2|\beta|-8}w_l(\al,\beta)\partial^{e_i}\phi\partial^{\alpha}_{\beta}f_\pm e^{\frac{\pm\phi}{2}},(\theta^w\partial_{e_i}(w_l(\al,\beta)){(\partial^\alpha_\beta f_\pm e^{\frac{\pm\phi}{2}})^\wedge})^\vee\big)_{L^2_{v,x}}\big|\\
&\lesssim \|\partial^{e_i}\phi\|_{H^2_x}\|\psi_{|\alpha|+|\beta|-4}w_l(\al,\beta)\partial^{\alpha}_{\beta}f_\pm\|_{L^2_{v,x}}^2\\
&\lesssim \delta_0\E_{K,l}(t),
\end{align*}with the help of \eqref{13} and $\theta\in S(1)$. 
When $\alpha_1\neq 0$, then $\alpha\neq 0$, the total number of derivatives on the first $f_\pm$ is less or equal to $K$ and there's at least one derivative on the second $f_\pm$ with respect to $x$. Thus, 
\begin{align*}
	|K_3|\lesssim \E^{1/2}_{K,l}\D_{K,l}.
\end{align*}
For $K_4$, there's exponential decay in $v$ and hence 
	$|K_4|\lesssim \E_{K,l}. $
For $K_5$, recalling that we only need an upper bound and using Lemma \ref{lemmat} with $\pa^\al\mu=0$ for $|\al|\ge 1$, we have $|K_5|\lesssim\E_{K,l}+\D_{K,l}.$
For $K_6$, we use Lemma \ref{lemmag} to obtain
\begin{align*}
	|K_6|\lesssim\E^{1/2}_{K,l}\D_{K,l}+\E_{K,l}\D^{1/2}_{K,l}\lesssim (\E^{1/2}_{K,l}+\E_{K,l})\D_{K,l}+\E_{K,l}.
\end{align*}
For $K_7$, since $\beta_1\neq 0$, one has $|\partial_{\beta_1}v_i|\lesssim 1$ and the total number of derivatives on the first $f_\pm$ is less or equal to $K$. Also, $w(\al,\beta)=\<v\>^\gamma w(|\al|+1,|\beta|-1)$. These yield that $|K_7|\lesssim \E_{K,l}$. Combining the above estimate with \eqref{104} and choosing $\delta>0$ sufficiently small, we have 
	\begin{align*}
		&\notag\quad\,\|\psi_{|\alpha|+|\beta|-4}\tilde{b}^{1/2}w_l(\al,\beta)(\partial^\alpha_\beta f)^\wedge(v,y)\|^2_{L^2_{v,y}}\\
		&\lesssim \frac{1}{2}\partial_t\big(-\psi_{2|\alpha|+2|\beta|-8}w_l(\al,\beta)\partial^\alpha_\beta f_\pm e^{\frac{\pm\phi}{2}},(\theta^ww_l(\al,\beta){(\partial^\alpha_\beta f_\pm e^{\frac{\pm\phi}{2}})^\wedge})^\vee\big)_{L^2_{v,x}}\\
		&\notag\qquad+(\E^{1/2}_{K,l}+\E_{K,l})\D_{K,l}+C_\delta\|\<v\>^{C_{K,l}}f\|_{L^2_{v,x}}^2+\|\partial_t\phi\|_{L^\infty_x}\E_{K,l}(t)+\delta^2\D_{K,l}+\E_{K,l}.
	\end{align*}
Substituting this into \eqref{101a}, we have the desired estimate \eqref{112}. This completes the proof of Lemma \ref{Lem33}.  

\qe\end{proof}

\begin{proof}[Proof of Theorem \ref{lem51}]
Substituting \eqref{112} into \eqref{72}, we have that for $0<\delta<1$, 
\begin{multline*}
	\partial_t\E_{K,l}(t)+\lambda D_{K,l}(t)
	\\\lesssim \delta^2\sum_{|\alpha|+|\beta|\le K}\partial_t\big(-\psi_{2|\alpha|+2|\beta|-8}w_l(\al,\beta)\partial^\alpha_\beta f_\pm e^{\frac{\pm\phi}{2}},(\theta^ww_l(\al,\beta){(\partial^\alpha_\beta f_\pm e^{\frac{\pm\phi}{2}})^\wedge})^\vee\big)_{L^2_{v,x}}\\
	\notag\qquad+ \|\partial_t\phi\|_{L^\infty_x}\E_{K,l}(t) +\delta^2\big(\D_{K,l}+(\E^{1/2}_{K,l}+\E_{K,l})\D_{K,l}\big)+\E_{K,l}+C_\delta\|\<v\>^{C_{K,l}}f\|_{L^2_{v,x}}^2,
\end{multline*}
By \eqref{34} and \eqref{priori1}, we have $\|\partial_t\phi\|_{L^\infty_x}\lesssim \E^{1/2}_{K,l}\lesssim \delta^{1/2}_0$ 
 Using the $a$ $priori$ assumption \eqref{priori1} and choosing $\delta,\delta_0>0$ sufficiently small, we have 
\begin{multline*}
	\partial_t\E_{K,l}(t)+\lambda \D_{K,l}(t)\lesssim
	\E_{K,l}(t)+\|\<v\>^{C_{K,l}}f\|_{L^2_{v,x}}^2\\
	+\delta^2\sum_{|\alpha|+|\beta|\le K}\partial_t\big(-\psi_{2|\alpha|+2|\beta|-8}w_l(\al,\beta)\partial^\alpha_\beta f_\pm e^{\frac{\pm\phi}{2}},(\theta^ww_l(\al,\beta){(\partial^\alpha_\beta f_\pm e^{\frac{\pm\phi}{2}})^\wedge})^\vee\big)_{L^2_{v,x}}.
\end{multline*}
By solving this ODE with neglecting $\lambda\D_{K,l}(t)$ and noticing 
\begin{align*}
	\big|\big(-\psi_{2|\alpha|+2|\beta|-8}w_l(\al,\beta)\partial^\alpha_\beta f_\pm e^{\frac{\pm\phi}{2}},(\theta^ww_l(\al,\beta){(\partial^\alpha_\beta f_\pm e^{\frac{\pm\phi}{2}})^\wedge})^\vee\big)_{L^2_{v,x}}\big|
	\lesssim \E_{K,l}(t),
\end{align*}
 we have that for $0\le t\le t_0$,
\begin{align}\notag
	\E_{K,l}(t) &\lesssim \E_{K,l}(0)+\delta^2\E_{K,l}(t)+\delta^2\E_{K,l}(0)+\int^t_0(\E_{K,l}+\|\<v\>^{C_{K,l}}f\|_{L^2_{v,x}})\,d\tau,\\
	\E_{K,l}(t)&\lesssim \epsilon^2_1,\label{76}
\end{align}by choosing $\delta>0$ and $t_0=t_0(\epsilon_1,\|\<v\>^{C_{K,l}}f\|_{L^2_{v,x}})>0$ sufficiently small. 
Here we used $\E_{K,l}(0)\le \E_{4,l}(0)$. 
This completes the proof of Theorem \ref{lem51}. 

\qe\end{proof}

\begin{proof}
	[Proof of Theorem \ref{main2}] 
	We prove Theorem \ref{main2} in four steps. 
	
	\smallskip
	\noindent{\bf Step 1.}
It follows immediately from the $a$ $priori$ estimate \eqref{priori1} and Theorem \ref{lem51} that $$\sup_{0\le t\le t_0}\E_{K,l}\le C_{K,l}\epsilon^2_1$$ holds true for some small $t_0>0$, as long as $\epsilon_1$ is sufficiently small. 
The rest is to prove the local existence and uniqueness of solutions in terms of the energy norm $\E_{K,l}$. One can use the iteration on system 
	\begin{equation*}
		\left\{\begin{aligned}
			&\partial_tf^{n+1}_\pm+v\cdot\nabla_xf^{n+1}_\pm\mp\nabla_x\phi^n\cdot\nabla_xf^{n+1}_\pm\pm\frac{1}{2}\nabla_x\phi^n\cdot vf^{n+1}_\pm\\&\qquad\qquad\pm v\mu^{1/2}\cdot\nabla_x\phi^n-L_\pm f=\Gamma_\pm(f^n,f^{n+1}),\\
			&-\Delta_x\phi^{n+1}=\int_{\R^3}\mu^{1/2}(f^{n+1}_+-f^{n+1}_-)\,dv,\\
			&f^{n+1}|_{t=0}=f_0,
		\end{aligned}\right.
	\end{equation*}
to find the local existence and the details of proof are omitted for brevity; see \cite{Guo2012, Strain2013} and \cite{Gressman2011}.

%		Notice that the constants in Lemma \ref{lem51} are independent of time $t$ and hence, we can apply Theorem \ref{lem51} to any time interval with length less than $t_0$ to obtain that for $0<\tau<T$, 
%		\begin{align}\label{106}
%			\sup_{\tau\le t\le T}\E_{K,l}(t)\lesssim \epsilon_1^2C_{T}. 
%		\end{align}
%	Recalling Definition \eqref{Defe} of $\E_{K,l}$ and the choice \eqref{93} of $\psi$, we have that for any $0<\tau<T$ and $l\ge K\ge 0$,
%	\begin{align}\label{79}
%	\sup_{\tau\le t\le T}\sum_{|\alpha|+|\beta|\le K}\|w^{l-|\alpha|-|\beta|}\partial^\alpha_\beta f\|^2_{L^2_{v,x}}+\sup_{\tau\le t\le T}\sum_{|\alpha|\le K}\|\partial^\alpha\nabla_x\phi\|_{L^2_x}^2\le  C_{\tau,T}<\infty.
%\end{align}	Notice that $\psi_{|\alpha|+|\beta|-4}^{-1}$ is singular near $t=0$ when $|\alpha|+|\beta|>4$, so the constant is necessarily depending on $\tau$. 
%This proves \eqref{19a}.

\smallskip
\noindent{\bf Step 2.}
	Notice that the constants in Lemma \ref{lem51} are independent of time $t$ and hence, we can apply Theorem \ref{lem51} to any time interval with length less than $t_0$ to obtain that, for $0<\tau<T$, 
\begin{align}\label{106}
	\sup_{\tau\le t\le T}\E_{K,l}(t)\le \epsilon_1^2C_{\tau,T,K,l}. 
\end{align}
Recalling Definition \eqref{Defe} of $\E_{K,l}$ and the choice \eqref{93} of $\psi$, we have that, for any $0<\tau<T$, $l\ge 0$ and $K\ge 4$,
\begin{align}\label{79}
	\sup_{\tau\le t\le T}\sum_{|\alpha|+|\beta|\le K}\|w_l(\al,\beta)\partial^\alpha_\beta f\|^2_{L^2_{v,x}}+\sup_{\tau\le t\le T}\sum_{|\alpha|\le K}\|\partial^\alpha\nabla_x\phi\|_{L^2_x}^2\le  C_{\tau,T,l}<\infty.
\end{align}	Notice that $\psi_{|\alpha|+|\beta|-4}^{-1}$ is singular near $t=0$ when $|\alpha|+|\beta|>4$, so the constant is necessarily depending on $\tau$. 
This proves \eqref{19a}.

	\smallskip
Let $l\ge 0$, $K\ge 4$ and assume additionally $\E_{4,C_{K,l}}(0)$ is sufficiently small for some large constant $C_{K,l}>0$ to be chosen later. Then by \eqref{106}, we have 
\begin{align}\label{107a}
	%\sup_{\tau\le t\le T}\sum_{|\alpha|+|\beta|\le K}\|w^{l-|\alpha|-|\beta|}\partial^\alpha_\beta f\|^2_{L^2_{v,x}}\le C_{\tau,T}.
	\sup_{\tau\le t\le T}\E_{K,C_{K,l}}(t)\le \epsilon_1^2C_{\tau,T,K,l}.
\end{align}
%
%If additionally $\sup_{l^*}\E_{4,l^*}(0)>0$ is sufficiently small. Then for $l_0\ge K\ge 4$, by \eqref{19a}, we have 
%\begin{align*}
%\sup_{\tau\le t\le T}\sum_{|\alpha|+|\beta|\le K}\|w^{l_0-|\alpha|-|\beta|}\partial^\alpha_\beta f\|^2_{L^2_{v,x}}\le C_{\tau,T}.
%\end{align*}
For the regularity on $t$, the technique above is not applicable and we only make a rough estimate. For any $t>0$, applying $\<v\>^l\partial^k_t\partial^\alpha_\beta$ with $k,l\ge 0$, $|\alpha|+|\beta|\le K$ to equation \eqref{7} and taking $L^2_{v,x}$ norms, we have   
\begin{multline}\label{105}
\|\<v\>^l\partial^{k+1}_t\partial^\alpha_\beta f_\pm\|^2_{L^2_{v,x}}
\lesssim \|\<v\>^lv\cdot\nabla_x\partial^k_t\partial^\alpha_\beta f_\pm\|^2_{L^2_{v,x}}+\|\<v\>^l\sum_{k_1\le k}\partial^{\alpha}_\beta\big(\partial^{k_1}_t\nabla_x\phi\cdot v\partial^{k-k_1}_tf_\pm\big)\|^2_{L^2_{v,x}}
\\
\qquad+\|\<v\>^l\sum_{k_1\le k}\partial^\alpha\big(\partial^{k_1}_t\nabla_x\phi\cdot\nabla_v\partial^{k-k_1}_t\partial_\beta f_\pm\big)\|_{L^2_{v,x}}^2+\|\<v\>^l\partial^k_t\partial^\alpha\nabla_x\phi\cdot \partial_\beta(v\mu^{1/2})\|^2_{L^2_{v,x}}\\
\qquad+\|\<v\>^l\partial^\alpha_\beta L_\pm \partial^k_tf_\pm\|^2_{L^2_{v,x}} + \|\<v\>^l\sum_{k_1\le k}\partial^\alpha_\beta\Gamma_\pm(\partial^{k_1}_tf,\partial^{k-k_1}_tf)\|_{L^2_{v,x}}^2. 
\end{multline}
Denoting $$\E_{K,l,k}=\sum_{|\alpha|+|\beta|\le K,k_1\le k}\|\<v\>^l\partial^\alpha_\beta\partial^{k_1}_tf\|_{L^2_{v,x}},$$ we estimate the right-hand terms of \eqref{105} one by one. The first term on the right hand is bounded above by $\E_{K+1,l+1,k}$.
%For terms involving both $\phi$ and $f_\pm$, we use \eqref{13} to generate one more $x$ derivative on $\phi$. 
Applying the trick in Lemma \ref{Lem26}, the second right hand term of \eqref{105} is bounded above by 
\begin{align*}
\sum_{|\alpha|+|\beta|\le K,\,k_1\le k}\|\partial^{k_1}_t\partial^\alpha_\beta\nabla_x\phi\|^2_{L^2_x}\sum_{|\alpha|+|\beta|\le K,\,k_1\le k}\|\<v\>^{l+1}\partial^{k_1}_t\partial^\alpha_\beta f_\pm\|^2_{L^2_{v,x}}\lesssim \E_{K,l+1,k}^2.
\end{align*}
Similarly, applying the trick in Lemma \ref{Lem27}, the third term of \eqref{105} is bounded above by $\E_{K+1,l+1,k}^2.$ For the fourth term, when $k=0$, it's bounded above by $\E_{K,l,0}$. When $k\ge 1$, by using \eqref{34a}, it's bounded above by $\E_{K,l,k-1}$. For the fifth term, noticing $L_\pm\in S(\tilde{a})\subset S(\<v\>^{\gamma+2s}\<\eta\>^{2s})$ and $s\in(0,1)$, we have 
\begin{align*}
\|\<v\>^l\partial^\alpha_\beta L_\pm \partial^k_tf_\pm\|^2_{L^2_{v,x}}\lesssim \|\<v\>^{l+{\gamma+2s}}\<D_v\>^2\<(D_x,D_v)\>^K \partial^k_tf_\pm\|^2_{L^2_{v,x}}\lesssim \E_{K+2,l+{\gamma+2s},k}.
\end{align*}
For the last term, using \eqref{12a}, it's bounded above by 
\begin{align*}
\sum_{|\alpha|+|\beta|\le K+2,\,k_1\le k}\|\<v\>^{l+\frac{\gamma+2s}{2}}\partial^\alpha_\beta\partial^{k_1}_tf\|^2_{L^2_{v,x}}\lesssim \E_{K+2,l+\frac{\gamma+2s}{2},k}^2. 
\end{align*} 
Combining the above estimate and taking summation of \eqref{105} over $|\alpha|+|\beta|\le K$, $k\le k_0$ for any $k_0\ge 0$, we have 
\begin{align*}
\E_{K,l,k_0+1}(t)&\lesssim \E_{K,l,0} +\E_{K,l,k_0-1}+\E_{K,l+1,k_0}+\E_{K+1,l+1,k_0}^2\\&\quad+\E_{K+2,l+{\gamma+2s},k_0}+ \E_{K+2,l+\frac{\gamma+2s}{2},k_0}^2.
\end{align*} 
The $t$ derivative on the right hand is less than the left hand. 
Hence, noticing \eqref{107a}, for any $T>\tau>0$, we have
\begin{align*}
\sup_{\tau\le t\le T}\E_{K,l,k_0}(t)\le C_{\tau,T,l,k_0}.
\end{align*}
For the time derivatives on $\na_x\phi$, we apply \eqref{34a} to obtain 
\begin{align*}
	\sup_{\tau\le t\le T}\sum_{|\al|\le K,\,k\le k_0}\|\pa^\al\pa^k_t\na_x\phi\|_{L^2_x}^2\lesssim \sup_{\tau\le t\le T}\E_{K,l,k_0}(t)\le C_{\tau,T,l,k_0}.
\end{align*}
Then we obtain \eqref{19b}. 
%The same standard argument for deriving the local solution gives \eqref{19b} and the details are omitted for brevity; see also \cite{Guo2012,Gressman2011} and \cite{Deng2020b}. 
Consequently, by Sobolev embedding, $f\in C^\infty(\R^+_t\times\R^3_x\times\R^3_v)$.

\medskip

\smallskip
\noindent{\bf Step 3.}
Now we additionally assume \eqref{19aa} is sufficiently small.
Noticing $\psi=1$ in Theorem \ref{thm21}, \eqref{15a} shows that for any $\tau_0\ge \tau$, 
\begin{align*}
\sum_{|\alpha|\le 4}\|\partial^\alpha E(\tau_0)\|^2_{L^2_x}+\sum_{|\alpha|\le 4}\|\partial^\alpha\P f(\tau_0)\|^2_{L^2_{v,x}}+\sum_{\substack{|\alpha|+|\beta|\le 4}}\|w_l(\al,\beta)\partial^\alpha_\beta(\I-\P) f(\tau_0)\|^2_{L^2_{v,x}}
\lesssim \epsilon_0^2.
\end{align*}
Using this as the initial data instead of \eqref{15aa}, we can apply the above calculation on any time interval $[\tau_0,\tau_0+t_0]$ to obtain the same estimate with constants independent of $T$. In this case, we use 
\begin{align*}
\sup_{\tau_0\le t\le \tau_0+t_0}\E_{K,l}(t)\lesssim \epsilon_1^2C_{\tau}
\end{align*}instead of \eqref{106}, where the constant $C_\tau$ is independent of $\tau_0$. Recall the choice of $$t_0=t_0(\epsilon_1,\|\<v\>^{C_{K,l}}f\|_{L^2_{v,x}})>0$$ in \eqref{76} such that $t_0$ is uniform in any time $t$, we can obtain a uniform estimate independent of time $T$ and this completes the proof of Theorem \ref{main2} (3). Notice that the estimate of \eqref{79} is necessarily depending on $\tau$ since $\psi_{|\alpha|+|\beta|-4}^{-1}$ is singular near $t=0$ when $|\alpha|+|\beta|>4$.

 \medskip

 \smallskip
 \noindent{\bf Step 4.}
 If we assume \eqref{19aab} is sufficiently small for some large enough $C_{K,l}>0$, then by Theorem \ref{thm21}, we can obtain the estimate \eqref{19aabb} with $n=C_{K,l}$. Then the result follows from the same argument as Step 3. 
\qe\end{proof}

	\medskip
\noindent {\bf Acknowledgements.} 
Dingqun Deng was supported by a Direct Grant from BIMSA and YMSC. The author would thank Prof. Tong Yang for the valuable comments on the manuscript. 

%\small
%\bibliographystyle{amsplain}
%\bibliography{1}

\providecommand{\bysame}{\leavevmode\hbox to3em{\hrulefill}\thinspace}
\providecommand{\MR}{\relax\ifhmode\unskip\space\fi MR }
% \MRhref is called by the amsart/book/proc definition of \MR.
\providecommand{\MRhref}[2]{%
	\href{http://www.ams.org/mathscinet-getitem?mr=#1}{#2}
}
\providecommand{\href}[2]{#2}

\end{document}